\documentclass[letterpaper, reqno,oneside]{amsart}

\usepackage{amsmath}
\usepackage{amsthm}
\usepackage{graphicx} 
\usepackage[letterpaper]{geometry}          
\usepackage[parfill]{parskip}    
\usepackage{graphicx}
\usepackage{amssymb}
\usepackage{epstopdf}
\usepackage{multicol}
\usepackage{hyperref}
\usepackage{cleveref}
\usepackage{MnSymbol}
\usepackage{bm}
\usepackage{verbatim}
\usepackage{soul}
\usepackage{pgf,tikz,tikz-cd}
\usepackage{mathrsfs}
\usetikzlibrary{arrows}
\usepackage{color}
\usepackage[sort&compress]{natbib}

\DeclareMathOperator{\PGL}{PGL}

\newtheorem*{theorem*}{Theorem}
\newtheorem*{question*}{Question}

\newtheorem{theorem}{Theorem}
\newtheorem{lemma}[theorem]{Lemma}

\newtheorem{definition}[theorem]{Definition}
\newtheorem{proposition}[theorem]{Proposition}
\newtheorem{corollary}[theorem]{Corollary}
\theoremstyle{remark}
\newtheorem{remark}{Remark}

\Crefname{question}{Question}{Question}

\newcommand{\rank}{\textup{rank}}

\newcommand{\RR}{\mathbb{R}}

\newcommand{\PP}{\mathbb{P}}
\newcommand{\CC}{\mathbb{C}}
\newcommand{\ZZ}{\mathbb{Z}}

\newcommand{\SL}{\textup{SL}}
\newcommand{\M}{\mathcal{M}}

\newcommand{\X}{\mathcal{X}}
\newcommand{\Y}{\mathcal{Y}}

\newcommand{\OO}{\mathcal{O}}
\newcommand{\sym}{\mathrm{Sym}}
\newcommand{\Pcal}{\mathcal{P}}
\newcommand{\Qcal}{\mathcal{Q}}

\usepackage[draft]{fixme}
\fxsetup{theme=color, mode=multiuser, noinline}
\FXRegisterAuthor{rt}{RT}{\color{blue}Rekha}
\FXRegisterAuthor{go}{GO}{\color{teal}Giorgio}

\title{When is one pinhole camera image equal to \\
some other pinhole camera image?}
\author{Giorgio Ottaviani}
\address{Dipartimento di Matematica e Informatica U. Dini, Universit\'a di Firenze, viale Morgagni 67/A, 50134 Firenze, Italy}
\email{giorgio.ottaviani@unifi.it}
\author{Rekha R. Thomas}
\address{Department of Mathematics, University of Washington,
Seattle, WA 98195, USA}
\email{rrthomas@uw.edu}
\subjclass[2020]{14N05,14L30,68T45} 
\date{\today}

\begin{document}
\begin{abstract}
Generically, one expects the images of two different point sets, in two different (projective) cameras, to be different. However, it can happen that the images are the same up to a projective transformation which is an instance of ill-posedness in computer vision. We prove that the images can become projectively equivalent only for point pairs with  at most seven elements. In each case, we give explicit descriptions of the Zariski closure of the 
locus of camera centers which we call the {\em centers-variety}. To do this we use classical invariant theory and the geometry of moduli spaces of ordered points in the projective plane. The most involved case is that of seven points which uses a natural parametrization of the Goepel variety. 
\end{abstract} 
\maketitle

\section{Introduction}
A {\em projective camera} is a linear map from $\PP^3 \dashrightarrow \PP^2$ that is represented (up to scale) by a matrix $A \in \RR^{3 \times 4}$ of rank three. This is a mathematical formalization (and generalization) of the traditional {\em pinhole camera}. A projective camera sends a {\em world point} $x \in \PP^3$ to its {\em image} $Ax \in \PP^2$. 
The {\em center} of the camera $A$, which is also the center of the projection map represented by $A$, is the unique element $a \in \PP^3$ such that $Aa = 0$. 
The image of a collection of world points 
$\X = \{x_1, \ldots, x_n\} \subset \PP^3$, under camera $A$, is $(Ax_i)_{i=1}^n \in (\PP^2)^n$. This image is well-defined if and only if $Ax_i \neq 0$ which holds 
if and only if $x_i \neq a$. Therefore, when we speak of camera images of $\X$, keep in mind that the only cameras in play are those whose centers do not lie in $\X$.

Typically one expects that if there are two sets of ordered world points 
\begin{align}
    \X = \{x_1, \ldots, x_n\} \subset \PP^3 \,\,\,\,\textup{ and } \,\,\,\,
    \Y = \{y_1, \ldots, y_n\} \subset \PP^3
\end{align}
and two projective cameras $A,B \,:\, \PP^3 \dashrightarrow \PP^2$, with centers $a$ and $b$, 
then $(Ax_i)_{i=1}^n \in (\PP^2)^n$, the image of $\X$ under $A$,  and $(By_i)_{i=1}^n \in (\PP^2)^n$, the image of $\Y$ under $B$, are distinct. 
However, it can happen that the images are the same up to a projective transformation, i.e., there exists $H \in \SL(3)$ such that $Ax_i = HBy_i$ for $i=1, \ldots, n$ where $=$ denotes equality in projective space. Note that $Ax_i$ can be any nonzero scaling of $ HBy_i$ with possibly different scalars for each $i=1, \ldots, n$. 
We denote $\SL(3)$-equivalence by $\sim$ and write 
$(Ax_i)_{i=1}^n \sim (Bx_i)_{i=1}^n$ if there exists $H \in \SL(3)$ such that $Ax_i = HBy_i$ for all $i$, and can ask 
when such an ill-posed situation arises. 

{\em Let $\X = \{x_1,\dots,x_n\}$ and 
$\Y = \{y_1, \dots, y_n\}$ be two sets of $n$ ordered points in $\PP^3$.  When are there projective cameras $A:\PP^3\dashrightarrow \PP^2$ and $B:\PP^3 \dashrightarrow \PP^2$ such that $(Ax_i)_{i=1}^n \sim (By_i)_{i=1}^n$?}

A specialization of Theorem 2 in \cite{flatlandpaper}  answers this question.

\begin{theorem} \cite[Theorem 2]{flatlandpaper} 
\label{thm:twocam-existence-general}
    Let $\mathcal{X} = \{x_1,\dots,x_n\} \subset \PP^3$ and 
    $\mathcal{Y} = \{y_1, \dots, y_n\} \subset \PP^3$ be two sets of ordered points. There exists full-rank linear projections  $A:\PP^3\dashedrightarrow \PP^{2}$ and $B:\PP^3 \dashedrightarrow \PP^{2}$ such that $(Ax_i)_{i=1}^n \sim (By_i)_{i=1}^n$ if and only if there exists full-rank linear projections $A':\PP^{4}\dashedrightarrow \PP^3$ and $B':\PP^{4} \dashedrightarrow \PP^3$ with centers $a',b'$, and a set of points $\mathcal{Z} = \{z_1,\dots,z_n\}\subset  \PP^{4}$, none lying in the linear span of $a'$ and  $b',$ such that
    $\forall i, \, x_i = A'z_i \text{ and } y_i = B'z_i.$
\end{theorem}

The triple $(\mathcal{Z},A',B')$ in~\Cref{thm:twocam-existence-general} is said to be a \emph{reconstruction} of $(\X , \Y )$.

Next we can ask for the locus of camera pairs $(A,B)$ for which $(Ax_i)_{i=1}^n \sim (By_i)_{i=1}^n$. 
To make this question precise we need to consider the action of $\SL(3)$ on both cameras and their images.
It is easy to check that a camera $A$ has center $a$ if and only if $HA$ has center $a$ for all $H \in \SL(3)$ \cite[Chapter 22.1]{hartley2003multiple}. Further, $y \sim Ax$ if and only if there exists $H \in \SL(3)$ such that $y=HAx$. Therefore, the $\SL(3)$-orbit of $Ax$ is precisely the set of images of $x$ under cameras in the $\SL(3)$-orbit of $A$. 

If we have a collection of world points  
$\X = \{x_1, \ldots, x_n\} \subset \PP^3$ then its image under camera $A$  
with center $a$ is $(Ax_i)_{i=1}^n \in (\PP^2)^n$,  and as above, the $\SL(3)$-orbit 
of $(Ax_i)_{i=1}^n$ consists of the images of $\X$ under all cameras with center $a$. Similarly, if $\Y = \{y_1, \ldots, y_n\} \subset \PP^3$ is another collection of world points imaged under camera $B$ with center $b$, then the $\SL(3)$-orbit of 
$(By_i)_{i=1}^n$ consists of the images of $\Y$ under all cameras with center $b$. The requirement that $(Ax_i)_{i=1}^n \sim (By_i)_{i=1}^n$ is exactly that the $\SL(3)$-orbits of $(Ax_i)_{i=1}^n$ and $(By_i)_{i=1}^n$ coincide. 
Since these image orbits are controlled by camera centers, and not camera matrices, the question we consider in this paper is the following. 
\vspace*{-0.15cm}
\begin{question*}
Given two sets of ordered points $\X = \{x_1,\dots,x_n\}$ and $\Y = \{y_1, \dots, y_n\}$  in $\PP^3$, what is the locus of cameras centers $(a,b) \in \PP^3 \times \PP^3$ such that 
$(Ax_i)_{i=1}^n \sim (By_i)_{i=1}^n$? 
(We call the Zariski closure of the locus of camera centers, the {\em centers-variety}.)
\end{question*}
\vspace{-0.15cm}
    Since the question is about the $\SL(3)$-equivalence of two images, which are necessarily well-defined, the locus of centers is in $ \left(\PP^3 \setminus \X \right) \times\left(\PP^3 \setminus \Y \right)$. However, its Zariski closure, which is the centers-variety, may well contain points in $(\X \times \PP^3) \cup (\PP^3\times\Y)$, and indeed it does for $n \leq 6$. 

{\bf The main results.} In this paper we give a complete answer to the above Question using classical invariant theory and algebraic geometry.
We summarize our main results at a high level in the following theorem. More nuanced statements and much more details can be found in \S~\ref{sec:n=5}--\ref{sec:n=7}.
\vspace{-0.1cm}
\begin{theorem}
Let $\X = \{x_1,\dots,x_n\}$ and $\Y = \{y_1, \dots, y_n\}$  be two sets of $n$ generic ordered points in $\PP^3$ such that there exist projective cameras $A$ and $B$ with centers $a$ and $b$ such that $(Ax_i)_{i=1}^n \sim (By_i)_{i=1}^n$.
    \begin{enumerate}
        \item If $n \leq 4$, the centers-variety is all of $\PP^3 \times \PP^3$ while if $n \geq 8$, generically the centers-variety is empty.
        \item If $n=5$ then the centers-variety is a rational $4$-fold in $\PP^3 \times \PP^3$. If $a$ is chosen generically, then $b$ lies on a unique twisted cubic curve that passes through $\Y$. A symmetric statement holds if $b$ is chosen generically. 
        \item If $n=6$ then the centers-variety is a smooth surface in $\PP^3 \times \PP^3$ that is isomorphic to a quadric in $\PP^3$ blown up at six points. Both camera centers $a$ and $b$ lie on quadric surfaces obtained as projections of the centers-variety. If $a$ is fixed then $b$ is uniquely determined, and vice-versa. 
        \item If $n=7$ then the centers-variety  consists of three points in $\PP^3 \times \PP^3$.
    \end{enumerate}
\end{theorem}

A simpler scenario is that of two sets of ordered points in $\PP^2$ being mapped by 
{\em flatland cameras} to the same image in $\PP^1$. The answers in this setting  also rely on classical invariant theory and algebraic geometry \cite{flatlandpaper}, and the paper describes connections of the Question to computer vision.

 {\bf Our central tool.} The  question naturally suggests the idea of using the invariant theory of ordered points in $\PP^2$ under the action of $\SL(3)$, and a moduli space of $n$ ordered points in $\PP^2$ modulo $\SL(3)$. Moduli spaces can be constructed using {\em Geometric Invariant Theory} \cite{mumford-invarianttheory} which allows for several options based on the choice of a {\em polarization}. We adopt the moduli space constructed in \cite{DolgachevGIT} and \cite{dolgachev-ortland} (denoted as $\M_2^n$), which uses the most natural polarization. This choice puts further restrictions on where the locus of camera centers can lie.  
To be precise, for $n= 6$ we consider $(a,b) \in \left(\PP^3\setminus\X\right)\times\left(\PP^3\setminus\Y\right)$ and 
the centers-variety is its Zariski closure in $\PP^3\times\PP^3$.
For $n=5,7$ we will also need to avoid all lines through pairs of world points when considering the locus of camera centers $(a,b)$. The precise details will be explained in the corresponding sections. 

{\bf Organization of the paper.}
In \S~\ref{sec:Mdn} we collect facts about moduli spaces of $n$ ordered points in $\PP^2$ with special emphasis on the cases of $n=5,6,7$. We also explain the phenomenon of {\em association}, which is a useful tool in this paper. 
In the remaining sections we answer the Question.

We begin in \S~\ref{sec:overview thm} with a birds-eye view of what to expect for the answer to the Question. Theorem~\ref{thm:loci} gives a lower bound of $14-2n$ for the dimension of centers-variety when $n=5,6,7$. In the remaining sections, we will see that the centers-variety has dimension exactly $14-2n$ when $n=5,6,7$. Theorem~\ref{thm:loci} also proves that the centers-variety is all of $\PP^3 \times \PP^3$ when $n=4$. 

In \S~\ref{sec:n=5} we prove that when $n=5$, the centers-variety is a rational $4$-fold in $\PP^3 \times \PP^3$. If one camera center is chosen generically,  then the other camera center lies on a unique twisted cubic. We give synthetic and computational proofs of this result, and an explicit description of the centers-variety. 

The case of $n=6$ point pairs is treated in \S~\ref{sec:n=6}. We prove that the centers-variety is isomorphic to a smooth quadric surface blown up in six points. Each camera center is constrained to lie on the projection of this surface into the corresponding $\PP^3$. If $a$ is chosen on its surface then $b$ is a unique point on its surface and vice-versa. Again, the 
centers-variety is described in detail. 

We finish in $\S~\ref{sec:n=7}$ by showing that the centers-variety consists of three distinct points when $n=7$ and that it is generically empty when $n \geq 8$. While in the previous sections we rely on a known set of invariants that can separate orbits, no such set is available in the literature when $n=7$. We bypass this difficulty by working with the {\em Goepel variety} which is birationally equivalent to the moduli space of seven ordered points in $\PP^2$, and identify a set of invariants to answer our Question.
A byproduct of our analysis is the explicit construction of 
$630={36 \choose 2}$ singular lines on the Goepel variety, forming a complete graph on $36$ vertices which are highly singular points. 

Our results rely on, or were inspired by, symbolic computations in the software package Macaulay2\cite{M2}. All computations can be found in the ancillary files of the arXiv submission of this paper. 

{\bf Acknowledgments.} We thank Timothy Duff for a crucial observation in the case of $n=7$ point pairs and we thank Sameer Agarwal and Igor Dolgachev for helpful discussions. 
Giorgio Ottaviani is a member of GNSAGA-INDAM and is supported by the European Union’s HORIZON–MSCA
2023-DN-JD program under the Marie Sklodowska-Curie
 Actions, grant agreement 101120296 (TENORS). Rekha Thomas 
 thanks the Department of Mathematics and Computer Science at the University of Florence for supporting her as a visitor in Autumn 2025.

\section{Moduli spaces of ordered points in $\PP^d$}\label{sec:Mdn}

\subsection{The GIT construction of the moduli space $\M_d^n$}
Geometric Invariant Theory (GIT) \cite{mumford-invarianttheory} allows one to construct moduli spaces of projective varieties under the action of  reductive groups. 
The moduli space $\M_d^n$ of $n$ ordered points in $\PP^d$, under the action of $\SL(d+1)$, can be constructed using GIT. We sketch the ideas and refer to \cite[Chapters I, II]{dolgachev-ortland}, 
\cite[Chapter 11]{DolgachevGIT} and \cite{mumford-invarianttheory} for details. A GIT-quotient requires the choice of a {\em polarization}, and in \cite{dolgachev-ortland}, Dolgachev and Ortland 
choose the minimal one which is homogeneous/democratic in the points.

\begin{definition}\label{def:l}Let $l$ be the smallest positive integer satisfying $ln=w(d+1)$ for some (positive) integer $w$.
\end{definition}
Note that when $d=2$, the case we are most interested in, the smallest $l$ for which $ln/3$ is a positive integer, is at most $3$. More precisely,  
\begin{align} 
l=\left\{\begin{array}{lr}1&\textrm{\ if\ } n=0\textrm{\ mod\ }3\\
3&\textrm{\ if\ } n\neq 0\textrm{\ mod\ }3\end{array}\right.
\end{align}
Consider the graded ring
\begin{align} 
R=\oplus_m (\underbrace{
\sym^{ml}\CC^{d+1}\otimes\ldots\otimes\sym^{ml}\CC^{d+1}}_{n\textrm{\ times}}) \,\,\textup{ with } \,\,\mathrm{Proj}(R)=\underbrace{\PP^d\times\ldots\times\PP^d}_{n\textrm{\ times}} = (\PP^d)^n.
\end{align}
The group $\SL(d+1)$ acts on $R$ as follows.
For $H \in \SL(d+1)$,  
$f \in R$ and $p=(p_1, \ldots, p_n) \in (\PP^d)^n$, 
\begin{align}
    (Hf)(p_1, \ldots, p_n) = f(H^{-1}p_1, \ldots, H^{-1}p_n). 
\end{align}
The polynomial $f$ is {\em invariant} under this action if 
$Hf = f$ for all $H$. The set of all invariant polynomials is a subring of $R$. This invariant subring $R^{\SL(d+1)} =:R^n_d$ inherits the grading in $R$ 
and is the homogeneous coordinate ring  of $n$ ordered points in $\PP^d$ under the action of $\SL(d+1)$.  Therefore, the moduli space of $n$ ordered points in $\PP^d$ under $\SL(d+1)$-action is 
$$\M_d^n := \mathrm{Proj}(R^n_d).$$

The first fundamental theorem of invariant theory (\cite[Prop. 2 pp. 9]{dolgachev-ortland}, \cite[Theorem 28]{OttavianiLectures}) claims that $R_d^n$ is generated by {\em bracket functions} (sometimes also called {\em standard monomials}). Readers who are unfamiliar with GIT
may be helped by the following examples of $\M_2^5$, $\M_2^6$ and $\M_2^7$, the moduli spaces of $n=5,6,7$ ordered points in $\PP^2$, modulo the action of $\SL(3)$.

\subsubsection{{The moduli space $\M_2^5$}} 
When $n=5$, the integer $l$ from 
Definition \ref{def:l} is $3$, and so,
\begin{align}
R_2^5=\oplus_m (\underbrace{
\sym^{3m}\CC^{3}\otimes\ldots\otimes\sym^{3m}\CC^{3}}_{5\textrm{\ times}})^{\SL(3)}.
\end{align} 
It is a nontrivial fact that the dimension of the $m$-th summand of $R_2^5$ 
is 
\begin{equation}\label{eq:rr5}
\frac{1}{2}(5m^2+5m+2).\end{equation}
There are several indirect proofs of this formula in the literature; it is the Hilbert polynomial of the quintic Del Pezzo surface in its anticanonical embedding.

The first summand of $R_2^5$ (when $m=1$) is $6$-dimensional, and a generating set for it is given by 
\begin{equation}\label{eq:firstgen5points}g_0=[124][134][135][235][245]\end{equation}
and five permutations of it. 
The notation $[ijk]$ denotes the bracket (determinant)
\begin{align} 
\det \begin{bmatrix}p_{i0}&p_{i1}&p_{i2}\\
p_{j0}&p_{j1}&p_{j2}\\
p_{k0}&p_{k1}&p_{k2}\end{bmatrix}
\end{align}
where $p_t = (p_{t 0}:p_{t 1}:p_{t 2})$
is the $t$th point in the collection of five points in $\PP^2$.
For example, we could take the following five permutations of $g_0$ to obtain the generating set $\{ g_0, \ldots, g_5 \}$. 
\begin{align} \label{eq:gens5points}
    \begin{array}{c|c}
\textup{permutation} & \textup{invariant}\\
\hline
(45) & [125][135][134][234][245] = g_1\\
(34) & [123][134][145][245][235] = g_2\\
(345) & [125][145][134][234][235] = g_3\\
(354) & [123][135][145][245][234] = g_4\\
(35) & [124][145][135][235][234] = g_5
\end{array}
\end{align}
Note that each point $p_t$ appears exactly three times in each $g_i$, and this is the number $l$ chosen in Definition \ref{def:l}. Hence the bracket functions $g_i$ are multihomogeneous of degree $(3,3,3,3,3)$ in the coordinates of the five points.

The six functions $g_0, \ldots, g_5$ generate all of $R^5_2$ and it is nontrivial that five quadratic (Pluecker) relations
generate the relations among them. Indeed, 
$\M_2^5$ is isomorphic to the variety
of $4$-Pfaffians of a $5\times 5$ matrix with general linear entries in six homogeneous variables. This is known to be the Del Pezzo surface of degree $5$, i.e., a surface obtained by blowing up four points in $\PP^2$, no three of which are in a line \cite[pp 31]{dolgachev-ortland}, and $(\ref{eq:rr5})$ is the Riemann-Roch theorem applied to this surface.

The Hilbert series of $R_2^5$ is
\begin{align}
    \begin{split}
        \sum_{m=0}^\infty\frac{1}{2}(5m^2+5m+2)t^m = \sum_{m=0}^\infty\left(5{{m+2}\choose 2}-5(m+1)+1\right)t^m \\
        =\frac{5}{(1-t)^3}-
\frac{5}{(1-t)^2}+\frac{1}{1-t}=
\frac{1+3t+t^2}{(1-t)^3}=
\frac{1-5t^2+5t^3-t^5}{(1-t)^6}
    \end{split}
\end{align}
where the coefficients of the numerator in the last expression are the alternating Betti numbers in the resolution of $\M_2^5$. In Macaulay2 \cite{M2} notation, they are displayed as $$\begin{matrix}
       & 0 & 1 & 2 & 3\\
      \text{total:} & 1 & 5 & 5 & 1\\
      0: & 1 & . & . & .\\
      1: & . & 5 & 5 & .\\
      2: & . & . & . & 1
      \end{matrix}$$
The denominator represents the six bracket generators $g_0, \ldots, g_5$. 

\subsubsection{The moduli space $\M_2^6$} \label{subsec:M26}
The moduli space $\M_2^6$ is simpler than $\M_2^5$
since the integer $l$ from Definition \ref{def:l} is $1$, and therefore, 
$$R_2^6=\oplus_m (\underbrace{
\sym^{m}\CC^{3}\otimes\ldots\otimes\sym^{m}\CC^{3}}_{6\textrm{\ times}})^{\SL(3)}.$$
By \cite[Chapter I.3, Example 3]{dolgachev-ortland}, the dimension of the $m$-th summand of $R_2^6$ is
\begin{equation}\label{eq:rr6}\frac{1}{12}(m^4+6m^3+17m^2+24m)+1.\end{equation}
The Hilbert series of $R_2^6$, computed in (\cite[Chapter I.3, Example 3]{dolgachev-ortland}, is 
\begin{align}
   \sum_{m=0}^\infty\left(\frac{1}{12}(m^4+6m^3+17m^2+24m)+1\right)t^m = 
\frac{1-t^4}{(1-t)^5(1-t^2)} 
\end{align}
which suggests that the moduli space $\M_2^6$ is isomorphic
 to a hypersurface of dimension $4$ and degree $4$ in the weighted projective space $\PP(1,1,1,1,1,2)$.
 Indeed, the invariant ring $R_2^6$
 is generated by the following five bracket functions of degree $1$: 
 \begin{align} \label{eq:degree 1 gens R26}
 t_0=[123][456], \quad t_1=[124][356], \quad
t_2=[125][346], \quad t_3=[134][256], \quad
t_4=[135][246]
 \end{align}
 and the following bracket function of degree $2$: 
 \begin{align} \label{eq:degree 2 gen R26}
t_5=[123][145][246][356]-[124][135][236][456], 
\end{align}
see \cite[Chapter I.3, Example 3]{dolgachev-ortland} or \cite[Example 11.7]{DolgachevGIT}. The six generators satisfy a relation $t_5^2=F(t_0,\ldots, t_4)$, where the 
polynomial $F$ is of degree $4$ and vanishes exactly when the six points lie on a conic. The involution
 $t_5\mapsto -t_5$ is called a {\em Gale transform}. See \S~\ref{sec:n=6} for more details. 

\subsubsection{The moduli space $\M_2^7$} When $n=7$, 
the integer $l$ in
Definition \ref{def:l} is again $3$, and so,
$$R_2^7=\oplus_m (\underbrace{
\sym^{3m}\CC^{3}\otimes\ldots\otimes\sym^{3m}\CC^{3}}_{7\textrm{\ times}})^{\SL(3)}.$$
The degree $1$ summand of $R_2^7$ is $15$-dimensional. We have not found  an explicit formula for
the dimension of the $m$-th summand in the literature, which is likely computable but not trivial. A standard relaxation of this problem is to consider the subring of $R_2^7$ generated by the degree $1$ summand. This defines a  quotient of $\M_2^7$ called the {\em Goepel variety}, a $6$-fold embedded in $\PP^{14}$, and is enough to describe and separate general orbits. The Goepel variety is also sufficient for our computer vision problem and we describe it in detail in \S~\ref{sec:goepel}.

\subsection{Stability in $\M_d^n$}
The inclusion 
$R^{\SL(d+1)} \subset R$, considered at the beginning of \S~\ref{sec:Mdn}, induces the quotient map 
$\mathrm{Proj}(R)\to \mathrm{Proj}(R^{\SL(d+1)})$ which we denote in our case by 
\begin{align}
\pi_{d,n}\colon\underbrace{\PP^d\times\ldots\times\PP^d}_{n\textrm{\ times}}\dashrightarrow \M_d^n.
\end{align}
The key point of GIT is to refine this map to obtain a parameter space for orbits with good properties. 
First of, $\pi_{d,n}$ is a rational map that is undefined exactly on the $n$-tuples at which all invariant functions vanish. An $n$-tuple $p=(p_1,\ldots, p_n)$ 
is {\em semistable} if at least one invariant polynomial is nonzero at $p$. The set of semistable $n$-tuples is denoted as
$$(\underbrace{\PP^d\times\ldots\times\PP^d}_{n\textrm{\ times}})^{ss}.$$ 
Note that $\pi_{d,n}$ is well defined (it is a morphism) exactly on the semistable $n$-tuples.
The set of semistable points contains the set of {\em stable} $n$-tuples, (see \cite{dolgachev-ortland}, \cite{mumford-invarianttheory} for details); those whose orbits are closed and whose stabilizers are finite. Two stable orbits can be distinguished by invariants or, in other words, for a stable $p=(p_1,\ldots, p_n)$ 
the fiber $\pi_{d,n}^{-1}\pi_{d,n}(p)$ contains a single stable orbit. The {\em Hilbert-Mumford criterion} \cite[Theorem 2.1]{mumford-invarianttheory} allows one to recognize stable and semistable points. The following result proven in 
\cite[Chapter II.2, Theorem 1, pp. 23]{dolgachev-ortland} is obtained by applying the Hilbert-Mumford criterion in our setting. 

\begin{theorem} \label{thm:Hilbert-Mumford criterion}
An $n$-tuple $(p_1,\ldots, p_n)\in 
\underbrace{\PP^d\times\ldots\times\PP^d}_{n\textrm{\ times}}$ 
is semistable if and only if for any proper subset $\{i_1,\ldots, i_k\}\subseteq
\{1,\ldots, n\}$, letting $\langle p_{i_1},\ldots, p_{i_k}\rangle$ denote the projective span of $p_{i_1}, \ldots, p_{i_k}$, 
\begin{align}\label{eq:ineqss}
\dim\langle p_{i_1},\ldots, p_{i_k}\rangle +1\ge\frac{k(d+1)}{n}.
\end{align}
Stable $n$-tuples are those for which the inequality is strict. 
\end{theorem}

 By \cite[Corollary, pp. 25]{dolgachev-ortland}, 
 semistability is equivalent to stability when $\textup{gcd}(n,d+1)=1$. Together with \cite[Theorem 2, pp. 28]{dolgachev-ortland}, we get that when $n \geq d+2$ and $\textup{gcd}(n,d+1)=1$, 
  the moduli space $\M_d^n$ is smooth and of dimension $d(n-d-2)$.
For any $n\ge d+2$, stable $n$-tuples project to smooth points of $\M_d^n$ under $\pi_{d,n}$.
In this paper we are interested in $\M_2^n$ whose 
dimension is $2n-8$ for $n\ge 4$. One can see this dimension count 
directly from the fact that $\M_2^n$ is the quotient of $(\PP^2)^n$ by $\SL(3)$; there are $2n$ parameters in $(\PP^2)^n$ and $8$ in 
$\SL(3)$. 

Another consequence of Theorem~\ref{thm:Hilbert-Mumford criterion} is that it is possible to give a geometric characterization of stable and semistable $n$-tuples in $(\PP^2)^n$ under $\SL(3)$-action. 

 \begin{corollary} \cite[Chapter VII, pp. 114]{dolgachev-ortland} For $p=(p_1,\ldots, p_5)\in\underbrace{\PP^2\times\ldots\times\PP^2}_{5\textrm{\ times}} $,  the following conditions are equivalent
 \begin{itemize}
 \item $p$ is semistable.
 \item $p$ is stable.
 \item The points $p_1,\ldots, p_5$ are distinct and 
 at most $3$ among them are collinear.
 \end{itemize}
 \end{corollary}

 \begin{corollary} \cite[Chapter VII.3, pp. 116]{dolgachev-ortland} For $p=(p_1,\ldots, p_6)\in\underbrace{\PP^2\times\ldots\times\PP^2}_{6\textrm{\ times}} $, 
 \begin{itemize}
 \item $p$ is semistable if and only if
at most $2$ among $p_1,\ldots, p_6$ coincide and at most $4$ are collinear.
\item $p$ is stable if and only if  $p_1,\ldots, p_6$ are distinct and at most $3$ among them are collinear.
 \end{itemize}
 \end{corollary}

In particular, six distinct points on a smooth conic form a stable configuration.

\begin{corollary}\label{coro:stable7} \cite[Chapter VII.4, pp. 120]{dolgachev-ortland} For $p=(p_1,\ldots, p_7)\in\underbrace{\PP^2\times\ldots\times\PP^2}_{7\textrm{\ times}} $, the following conditions are equivalent
 \begin{itemize}
 \item $p$ is semistable.
 \item $p$ is stable.
 \item at most $2$ points among $p_1,\ldots, p_7$ coincide and at most $4$ are collinear.
 \end{itemize}
 \end{corollary}

 In particular, seven points on a smooth conic, among which at most two coincide, make a stable configuration.

Suppose  $p = (Ax_i)_{i=1}^n \in (\PP^2)^n$ is the image of $\X = \{x_1, \ldots, x_n\} \subset \PP^3$ under camera $A$ with center $a$. Then we can ask for conditions on $a$ such that $p$ is semistable  since those are the only types of images that can be handled via $\M_2^n$. The above characterizations of semistability help us determine the indeterminacy locus of camera centers in this sense. We use the notation $2\X$ to mean $\X + \X$, i.e., lines that join pairs of points in $\X$.

 \begin{corollary}\label{coro:indeterm}
 Let $\X=\{x_1,\ldots, x_n\}\subset\PP^3$ be generic and $n\ge 5$. The indeterminacy locus of the rational map
 $$\begin{array}{ccc}\PP^3&\dashrightarrow&\M^n_2\\
 a&\mapsto&[(Ax_i)_{i=1}^n] \end{array}$$
 is given by $\left\{\begin{array}{cc}
    \X  & \textrm{\ for\ }n\ge 6 \\
    2\X  & \textrm{\ for\ }n = 5\end{array}\right.$.
   \end{corollary}

\begin{proof}
We will argue that the inequality (\ref{eq:ineqss}) for
  $(Ax_i)_{i=1}^n \in (\PP^2)^n$ is 
satisfied if $a\notin\X$ for $n\ge 6$ and if $a\notin 2\X$ for $n=5$.
Note that (\ref{eq:ineqss}) is trivially true when $k=1$. 

Suppose $n \geq 6$. By the genericity of $\X$,
$\dim\langle Ax_{i_1},\ldots, Ax_{i_k}\rangle\ge (k-2)$ for
$2\le k\le 4$ and it is equal to $2$ for $k\ge 4$.
Therefore, 
$\dim\langle Ax_{i_1},\ldots, Ax_{i_k}\rangle+1\ge k-1\ge \frac {3 k}{6}\ge
\frac{3 k}{n}$ for $k=2, 3, 4$ which proves (\ref{eq:ineqss}).

For any $n$, and $a\in\PP^3 \setminus \X$,
  $Ax_i, i=1,\ldots n$ are $n$ points in $\PP^2$,
  which, by the genericity of $\X$, are distinct if and only if $a\notin 2\X$. If $n=5$ and $a \not \in 2\X$, we get
$$\dim\langle Ax_{i_1},\ldots, Ax_{i_k}\rangle+1=\left\{
\begin{array}{cc}2&\textrm{\ for\ }k=2\\
\ge 2&\textrm{\ for\ }k=3\\
3&\textrm{\ for\ }k\ge 4\end{array}\right.$$
which proves (\ref{eq:ineqss}) again.
\end{proof}
\subsection{Association}\label{subsec:association}

There are natural isomorphisms
$$a_{d,n}\colon\M_d^n\to \M_{n-d-2}^n$$
which are involutive, i.e.
$a_{n-d-2,n}\cdot a_{d,n}=\mathrm{id}$, called {\em association} isomorphisms. We refer again to \cite[Chapter III]{dolgachev-ortland} for details. In the combinatorial setting they are called Gale transforms.

A brief description is as follows.
A general point in $\M_d^n$ can be represented by an $n$-tuple given by the columns of a $(d+1)\times n$
matrix of maximal rank. Its right kernel is spanned by the columns of a $n\times (n-d-1)$ matrix whose rows represent a point in $\M_{n-d-2}^n$; the class on the right so obtained, is the image of the class on the left under $a_{d,n}$. An $n$-tuple in $(\PP^d)^n$ is {\em self-associated} if it also represents the class on the right. This may happen only if $n=2d+2$ such as   
$6$ points in $\PP^2$.

Association plays a role in our setting. 
First, we have the association
$\M_2^6\to \M_2^6$. 
The self-associated $6$-tuples come from six points on a 
conic. 
The second important case for us is
\begin{align}\label{eq:a25}a_{2,5}\colon \M_2^5\to \M_1^5\end{align}
which, on the open part of $\M_2^5$ consisting of $5$-tuples such that {\em no three points are collinear}, can be described as follows,
 see \cite[Proposition 2, pp. 41]{dolgachev-ortland}. 
Given five general points in $\PP^2$,
there is a unique smooth conic through them.
Since the conic is isomorphic to $\PP^1$,
the five points determine an element of $\M_1^5$.
Another $5$-tuple in $(\PP^2)^5$ which is $\SL(3)$-equivalent to the first gives an $\SL(2)$-equivalent $5$-tuple on $\PP^1$. The construction is invertible because
every automorphism of a smooth conic comes as the restriction of an automorphism of $\PP^2$.

This description of $a_{2,5}$ can be generalized to
\begin{align}
    a_{n,n+3}\colon\M_n^{n+3}\to \M_1^{n+3}.
\end{align}
On the open part of $\M_n^{n+3}$ consisting of $(n+3)$-tuples such that no $n+1$ points lie on a hyperplane, the description is the following.  There is a unique rational normal curve of degree $n$ through the $n+3$ points \cite[Theorem 1.18]{HarrisAGbook}. Since the curve is isomorphic to $\PP^1$, as before, we get a $(n+3)$-tuple on $\PP^1$. Again every automorphism of a rational normal curve in $\PP^n$ comes as the restriction of an automorphism of $\PP^n$ which shows that 
$a_{n,n+3}$ is an isomorphism.

We end this section with a classical result needed later. 
For a modern reference see Kapranov's approach in 
\cite[Theorem 0.1 (a)]{Kap} along with its proof on \cite[pp. 243]{Kap}.

\begin{proposition}\label{prop:Kap}
    The variety of rational normal curves of degree $n$ through $n+2$ fixed points in $\PP^n$, such that no $n+1$ lie in a hyperplane,  is isomorphic to the open part of $\M_1^{n+2} = \M_{n-1}^{n+2}$ consisting of classes of distinct points in $\PP^1$.
\end{proposition}

\section{A dimension bound on the centers-variety}
\label{sec:overview thm}

We begin with some basic information on the dimension of the centers-variety, defined in the Question from the Introduction.

\begin{theorem}\label{thm:loci}
For two generic sets of ordered points  $\X = \{x_1,\dots,x_n\}$ and $\Y = \{y_1, \dots, y_n\}$ in $\PP^3$, the centers-variety is $\PP^3 \times \PP^3$ if $n \leq 4$, and 
 has dimension at least $14-2n$ if $n=5,6,7$.
\end{theorem}

\begin{proof} 
Let $\X$ and $\Y$ be two generic sets of world points with $n \leq 4$ points. Then for any pair of general cameras $A$ and $B$ with centers $a$ and $b$, $(Ax_i)_{i=1}^n$ and $(By_i)_{i=1}^n$ consist of points in $\PP^2$ in general position. Therefore, there is always a projective transformation taking one image to the other. The Zariski closure of all such general $(a,b)$ is $\PP^3 \times \PP^3$ which is then the centers-variety.

Suppose $n\ge 5$. Let $V_n^0$ denote all pairs 
$((a, \X), (b, \Y))\in\PP^3\times(\PP^3)^n\times
\PP^3\times(\PP^3)^n$ for which the projection of $\X$ through $a$ and the projection of $\Y$ through $b$ are defined, the images are 
 $\SL(3)$-equivalent and semistable, hence represented by the same point on $\M_2^n$. 
To be precise,  define
$$\X_a :=\left\{\begin{array}{ll}\{(a,\X)\in\PP^3\times(\PP^3)^n \, | \, a\in\X\}&\textrm{\ if\ }n\ge 6\\
\{(a,\X)\in\PP^3\times(\PP^3)^5 \, | \, a\in 2\X\}&\textrm{\ if\ }n=5\end{array}\right.$$
and similarly, $\Y_b$. By Corollary~\ref{coro:indeterm}, the  following fiber product defines $V_n^0$ as a subscheme of 
$\PP^3\times(\PP^3)^n\times
\PP^3\times(\PP^3)^n$:
\begin{center}
\begin{tikzcd}
& V_n^0 \arrow{dl}{}[swap]{} \arrow{dr}{} & \\
\left(\PP^3\times(\PP^3)^n\right)\setminus\X_a \arrow{dr}{}[swap]{A} & & \left(\PP^3\times(\PP^3)^n\right)\setminus\Y_b \arrow{dl}{B} \\
& \M^n_2 &
\end{tikzcd}
\end{center}
Since the maps $A$ and $B$ are both dominant, and 
the dimension of $\M^n_2$ 
is $2n-8$, 
we have 
\begin{align} 
\dim V_n^0\ge \dim(\PP^3\times(\PP^3)^n \times \PP^3\times(\PP^3)^n )-\dim\M^n_2 =
6(n+1)-(2n-8) = 4n+14.
\end{align}
Consider the two projections:
\begin{center}
\begin{tikzcd}
& V_n^0 \arrow{dl}{}[swap]{\pi_1} \arrow{dr}{}{\pi_2} & \\
\PP^3\times\PP^3  & & (\PP^3)^n\times(\PP^3)^n
\end{tikzcd}
\end{center}
The general fiber of $\pi_2$ has dimension at least $
4n+14-6n=14-2n$. This fiber is isomorphic to its image through $\pi_1$, whose closure is the centers-variety. This concludes the proof.
\end{proof}

Corollary~\ref{cor:n=8locus} will prove that the centers-variety is generically empty when $n \geq 8$. 
In the forthcoming sections we will see that the inequality in Theorem \ref{thm:loci} is indeed an equality for $n=5,6,7$ point pairs, and we give explicit descriptions of the centers-variety in each case. In the flatland situation studied in \cite{flatlandpaper}, the calculation in the proof of Theorem~\ref{thm:loci} shows that the centers-variety has dimension at least $4-(n-3) = 7-n$. The results in \cite{flatlandpaper} establish equality.

\section{$n=5$}\label{sec:n=5}

In the remaining sections we will describe the centers-variety 
for  $n=5,6,7$ point pairs. Since $\X$ and $\Y$ are sets of generic points, without loss of generality we may always assume that the first five points in both $\X$ and $\Y$ are the standard vectors in $\PP^3$:
\begin{align} \label{eq:std vectors}
    (1:0:0:0), \,\,(0:1:0:0),\,\, (0:0:1:0),\,\, (0:0:0:1),\,\, (1:1:1:1).
\end{align}
Indeed, if $(Ax_i)_{i=1}^n \sim (By_i)_{i=1}^n$ then also 
$((AU^{-1})(Ux_i))_{i=1}^n \sim ((BV^{-1})(Vy_i))_{i=1}^n$ for all $U,V \in \PGL(4)$. Therefore, we may apply projective transformations to $\X$ and $\Y$ before considering the Question we are interested in.

Our first result for $n=5$ is the following theorem. 

\begin{theorem} \label{thm:n=5locus}
Let $\X = \{x_1,\dots,x_5\}$ and $\Y = \{y_1, \dots, y_5\}$  be two generic sets of five ordered points in $\PP^3$. Consider two cameras $A,B$ with centers $a,b$ such that 
$(Ax_i)_{i=1}^5 \sim (By_i)_{i=1}^5$.
If $a$ is chosen generically,  then $b$ lies on a twisted cubic curve that contains $\Y$,  and symmetrically, if $b$ is chosen generically, then $a$ lies on a twisted cubic curve that contains $\X$. 
\end{theorem}

\begin{proof}
Without loss of generality, $\X=\Y$ consists of the five standard points  \eqref{eq:std vectors}.
Pick $a \in \PP^3$ generically and consider the unique twisted cubic curve $T_a$ through $\X \cup \{a\}$. (Recall that $a \not \in \X$ and that there exists a unique rational normal curve through $d+3$ generic points in $\PP^d$.) Projecting $\X$ through $a$, we obtain five points in $\PP^2$ lying on a smooth conic that is the image of $T_a$ under the projection. The five image points in $\PP^2$ correspond to a unique point $p \in \M_2^5$. Similarly if we pick $b$ generically, there is a unique twisted cubic curve $T_b$ through $\Y \cup \{b\}$ and the projection of $\Y$ through $b$ lies on a smooth conic in $\PP^2$ that is the image of $T_b$. 

By association, $\M_2^5 = \M_1^5$ and by Proposition~\ref{prop:Kap}, the variety of twisted cubics through $\X = \Y$ is isomorphic to the open part of  $\M_1^5$ corresponding to classes of five distinct points in $\PP^1$. Therefore, 
the two sets of projected points ($\X$ through $a$ and $\Y$ through $b$) can correspond to the same point in $\M_2^5$ if and only if they correspond to the same point in $\M_1^5$ if and only if $T_a=T_b$. 
We conclude that if $a \in \PP^3$ is picked generically then the locus of $b$ is the unique twisted cubic through $(\X=\Y) \cup \{a\}$. 
\end{proof}

As a consequence of the proof of Theorem~\ref{thm:n=5locus} we can describe the centers-variety when $n=5$. Recall that the notation 
$2\X$ stands for the set of all lines through pairs of points in $\X$.

\begin{corollary}\label{cor:4fold}
     When $n=5$, 
     the centers-variety is a $4$-fold in $\PP^3 \times \PP^3$ given by the closure of the fiber product
     \begin{align}
        \left(\PP^3\setminus 2\X\right) \times_{\M_2^5} \left(\PP^3\setminus 2\Y\right)
     \end{align}
described in the following diagram. 
\begin{center}
\begin{tikzcd}
& \left(\PP^3\setminus 2\X\right) \times_{\M_2^5} \left(\PP^3\setminus 2\Y\right)\arrow{dl}{}[swap]{} \arrow{dr}{} & \\
\PP^3\setminus 2\X\arrow{dr}{}[swap]{\alpha} & & \PP^3\setminus 2\Y \arrow{dl}{\beta} \\
& \M_2^5 &
\end{tikzcd}
\end{center}     
The projection
     $\alpha$ sends $a$ to the twisted cubic through
     $\X \cup \{a\}$, while the projection $\beta$
     sends $b$ to the twisted cubic through
    $\Y \cup \{b\}$ via Proposition \ref{prop:Kap} and association.
    \end{corollary}
\begin{proof}
   By Corollary~\ref{coro:indeterm}, we need to remove $2\X$ and $2 \Y$ from their corresponding $\PP^3$ before considering camera centers, to be able to map to $\M_2^5$. 
    As before, we may take $\X=\Y$, and then the images of $\X$ through $a$ and $b$ map to the same point in $\M_2^5 \cong \M_1^5$ if and only of the twisted cubics $T_a$ and $T_b$ are the same. Compare with the discussion after (\ref{eq:a25}). The dimension count follows since the centers-variety has a dominant map over $\PP^3$ with one dimensional general fibers. 
\end{proof}

\begin{corollary}\label{cor:minors}
\begin{enumerate}
    \item{} 
When $n=5$, 
     the centers-variety is the degeneracy locus of a
     vector bundle map on $\PP^3\times\PP^3$
     \begin{align}\label{eq:bundlemap}
\OO(-1,0)\oplus\OO(0,-1)\oplus\OO(-1,-1)\to\OO^4
\end{align}
 \item The multidegree of the centers-variety  
is $3a^2+5ab+3b^2$ in
the Chow ring of $\PP^3\times\PP^3$ which is $\ZZ[a,b]/(a^4,b^4)$. The centers-variety is a rational $4$-fold.
 \item If, without loss of generality,  $\X=\Y$ consists of the five standard points of $\PP^3$, then the ideal of the centers-variety consists of
the maximal minors of the matrix
\begin{equation}\label{eq:centersvar5}\left(\!\begin{array}{ccc}
      a_{0}&b_{0}&a_{0}b_{0}\\
      a_{1}&b_{1}&a_{1}b_{1}\\
      a_{2}&b_{2}&a_{2}b_{2}\\
      a_{3}&b_{3}&a_{3}b_{3}
\end{array}\!\right).\end{equation}
The matrix \eqref{eq:centersvar5} represents the map \eqref{eq:bundlemap}, in particular the three columns of \eqref{eq:centersvar5} correspond to the three summands of the source of \eqref{eq:bundlemap}.
\end{enumerate}
\end{corollary}
\begin{proof}
    We first prove (3).     
    Given a general $(a_0,\ldots, a_3)\in\PP^3$, we may perform row operations on the matrix \eqref{eq:centersvar5} to reduce the first column to have exactly one nonzero entry.      
    Then the locus of $(b_0,\ldots, b_3)\in\PP^3$ for which the matrix is rank deficient is the locus of maximal minors of the complementary $3\times 2$ matrix, which has linear entries in $b$. This locus  passes through $\X \cup \{a\}$ since evaluating  \eqref{eq:centersvar5} at $b\in\X\cup \{a\}$ makes two columns linearly dependent, and hence, 
    it is the unique twisted cubic through $\X \cup \{a\}$. By Corollary 
\ref{cor:4fold}, this proves (3), and (1) follows immediately. 
In order to prove (2), let $E :=\OO(-1,0)\oplus\OO(0,-1)\oplus\OO(-1,-1)$, then the Giambelli-Thom-Porteous formula gives that the class of 
the degeneracy locus of a map $E\to\OO^4$ is $c_2(\OO^4-E)$
which is the degree $2$ term of the development of the ratio of Chern polynomials $\frac{c(\OO^4)}{c(E)}$ which is $c_1^2(E)-c_2(E)$. We compute $c_1(E)=-2a-2b$, $c_2(E)=a^2+3ab+b^2$, which proves the multidegree.
The rationality follows since we may assume that $\X$ consists of the five standard points of $\PP^3$, then every twisted cubic through $\X \cup \{a\}$ can be
parametrized by the $\PP^1$-pencil of planes through $(1,0,0,0)$
and $(0,1,0,0)$. We get an injective dominant map from $\PP^3\times\PP^1$ to the centers-variety.
\end{proof}

\begin{remark}
The codimension $2$ cycles which generate the Chow ring
correspond to $a^2$, $ab$ and $b^2$. Their degrees in the Segre embedding of $\PP^3 \times \PP^3 \hookrightarrow \PP^{15}$ are respectively $4,6,4$, since
$a^2(a+b)^4=4a^3b^3$, $ab(a+b)^4=6a^3b^3$,
$b^2(a+b)^4=4a^3b^3$.
Hence the degree of the centers-variety 
in the Segre embedding is $3\cdot 4+5\cdot 6+3\cdot 4 = 54$.
The Hilbert polynomial of the centers-variety with respect to the line bundle $\OO(1,1)$ can be computed by (\ref{eq:centersvar5}) and is 
$$54P_4(t)-112P_3(t)+77P_2(t)-19P_1(t)+1,$$ where $P_i(t)={{t+i}\choose i}$.
\end{remark}

\begin{remark} The maximal minors of the obvious $(n+1)$-generalization of
(\ref{eq:centersvar5}) give the equations of the centers-variety obtained by projecting $n+2$ points from  $\PP^n$ to $\PP^{n-1}$, when $\X=\Y$ (without loss of generality) consists of the $n+2$ standard points of $\PP^n$.  Given a general $(a_0,\ldots, a_n)\in\PP^n$, the locus of $(b_0,\ldots, b_n)\in\PP^n$
    satisfying the generalization of (\ref{eq:centersvar5}) is the unique rational normal curve through $\X \cup \{a\}$. The argument is the same. The case $n=2$ corresponds to Theorem 26 in \cite{flatlandpaper}.
\end{remark}

\begin{remark}\label{rem:sing5}
    The singular locus of the $4$-dimensional centers-variety is given by the five diagonal points in $\X\times\Y$ and the following $50$ surfaces:
\begin{itemize}
\item The $30$ surfaces isomorphic to $\PP^1\times\PP^1$ given by
$\langle x_i, x_j\rangle\times\langle y_p, y_q\rangle$
with $\{i, j\}\cap\{ p, q\}=\emptyset$.

\item The $10$ surfaces isomorphic to $\PP^2\times\{p\}$ given by $\langle x_i, x_j, x_k\rangle\times \left(
\langle y_i, y_j, y_k\rangle\cap\langle y_p, y_q\rangle\right)$
with $\{i, j, k\}\cap\{ p, q\}=\emptyset$.

\item The $10$ surfaces isomorphic to $\{p\}\times\PP^2$ given by $\left(
\langle x_i, x_j, x_k\rangle\cap\langle x_p, x_q\rangle\right)\times
\langle y_i, y_j, y_k\rangle$ with $\{i, j, k\}\cap\{ p, q\}=\emptyset$.
\end{itemize}

Note that the surfaces isomorphic to $\PP^1\times\PP^1$ given by
$\langle x_i, x_j\rangle\times\langle y_p, y_q\rangle$
with $\{i, j\}\cap\{ p, q\}\neq\emptyset$ (therefore, different from those above) are not contained in the $4$-fold centers-variety.
This remark can be checked computationally, see the ancillary file \verb|"5pts-center_variety.m2"| 
in the arXiv submission. The fact that these surfaces consist of singular points can be seen
by degeneration arguments as in next Remark \ref{rem:cv5}. 
\end{remark}

Next we provide a computational proof of 
Theorem~\ref{thm:n=5locus} which is helpful in degenerate situations. 

{\em Computational proof of Theorem~\ref{thm:n=5locus}.}
Recall that the coordinate ring $R_2^5$ of $\M_2^5$ is generated by the six bracket polynomials $g_0, \ldots, g_5$ in \eqref{eq:firstgen5points} and \eqref{eq:gens5points}. 
We use this parametrization of $\M_2^5$ to prove Theorem~\ref{thm:n=5locus}. 

As before, we assume that $\X$ consists of the five standard points in $\PP^3$ \eqref{eq:std vectors}.
We project them through $a=(a_0:a_1:a_2:a_3)$ into $\PP^2$. This can be done by computing the intersection of the line $\ell_i = \langle a,x_i \rangle$ with the chart $x_3=0$ in $\PP^3$, and then dropping the last coordinate. The projection yields the $5$-tuple 
$\Pcal(a) = (p_1(a), \ldots, p_5(a)) \in (\PP^2)^5$ where 
\begin{align}
    \begin{split}
p_1(a) = (1:0:0), \,\,p_2(a) = (0:1:0), \,\,p_3(a) = (0:0:1), \\
p_4(a) = (a_0:a_1:a_2), \,\,p_5(a) =(-a_0+a_3:-a_1+a_3: -a_2+a_3).
\end{split}
\end{align}
We can then map $\Pcal(a)$ to $\M_2^5$ using the generators of $R_2^5$ to get the point 
\begin{align}
    g(\Pcal(a)) := (g_0(\Pcal(a)),\ldots, g_5(\Pcal(a))) \in \M_2^5.
\end{align} 
It can be checked (see the ancillary file \verb|5pts-b_locus.m2| of the arXiv submission) that 
\begin{small}
\begin{align*}
g_0(\Pcal(a))= & 
-a_0^2a_1^2a_2a_3+a_0a_1^2a_2^2a_3+a_0^2a_1a_2a_3^2+a_0a_1^2a_2a_3^2-a_0a_1a_2^2a_3^2-a_1^2a_2^2a_3^2-a_0a_1a_2a_3^3+a_1a_2^2a_3^3, \\
  g_1(\Pcal(a))= &-a_0^2a_1^2a_2a_3+a_0a_1^2a_2^2a_3+a_0^2a_1^2a_3^2+a_0^2a_1a_2a_3^2-a_0a_1^2a_2a_3^2-a_0a_1a_2^2a_3^2-a_0^2a_1a_3^3+a_0a_1a_2a_3^3,\\
 g_2(\Pcal(a))= &-a_0^2a_1^2a_3^2+a_0^2a_1a_2a_3^2+a_0a_1^2a_2a_3^2-a_0a_1a_2^2a_3^2+a_0a_1^2a_3^3-a_0a_1a_2a_3^3-a_1^2a_2a_3^3+a_1a_2^2a_3^3, \\
 g_3(\Pcal(a))= &-a_0^2a_1^2a_2a_3+a_0^2a_1a_2^2a_3+a_0^2a_1^2a_3^2-a_0^2a_1a_2a_3^2+a_0a_1^2a_2a_3^2-a_0a_1a_2^2a_3^2-a_0a_1^2a_3^3+a_0a_1a_2a_3^3,\\
 g_4(\Pcal(a))= &-a_0^2a_1^2a_3^2+a_0^2a_1a_2a_3^2+a_0a_1^2a_2a_3^2-a_0a_1a_2^2a_3^2+a_0^2a_1a_3^3-a_0^2a_2a_3^3-a_0a_1a_2a_3^3+a_0a_2^2a_3^3,\\ 
 g_5(\Pcal(a))= &-a_0^2a_1^2a_2a_3+a_0^2a_1a_2^2a_3+a_0^2a_1a_2a_3^2+a_0a_1^2a_2a_3^2-a_0^2a_2^2a_3^2-a_0a_1a_2^2a_3^2-a_0a_1a_2a_3^3+a_0a_2^2a_3^3\\
\end{align*}
\end{small}

As expected, all $g_i(\Pcal(a)$ will vanish if $a \in 2\X$. 
Since we can assume that 
 $y_i = x_i$, projecting $\Y$ through $b \in \PP^3$, we arrive at the point $g(\Pcal(b)) = (g_0(\Pcal(b)), \ldots, g_5(\Pcal(b))) \in \M_2^5$. 

By assumption, $g(\Pcal(a)) = g(\Pcal(b))$. If we fix $a$ randomly (say $a = \bar{a}$) then we get a (scalar) point $u = g(P({\bar{a}})) \in \M_2^5$. The fiber of $u$ under the map $\Pcal(b) \mapsto g(\Pcal(b))$ is the locus of $b$. To compute this preimage/fiber, we compute all $2 \times 2$ minors of 
\begin{align}
    \begin{pmatrix} g_0(\Pcal(b)) & \cdots & g_5(\Pcal(b)) \\ 
                    {u_0} & \cdots & u_5 \end{pmatrix}
\end{align}
and saturate with the ideal generated by $g_0(\Pcal(b)), \ldots, g_5(\Pcal(b))$. 

For example, if 
$\bar{a} = (43, -50, 6, -5)$ we get the fiber to be the variety of the ideal 
\begin{align}
\langle 
28b_1b_2+27b_1b_3-55b_2b_3, \,\,185b_0b_2+288b_0b_3-473b_2b_3, \,\,31b_0b_1-160b_0b_3+129b_1b_3
\rangle
      \end{align}
which has dimension $1$ and degree $3$. Since the variety is smooth and not contained in a plane, the fiber is a twisted cubic in $\PP^3$. One can check that this curve passes through $\Y$. \qed

The above computational strategy can be utilized to understand 
degenerate situations. 

\begin{remark}\label{rem:cv5}

\begin{itemize}
\item When the camera center $a$ is not coplanar with three of the points in $\X$ then the twisted cubic containing $b$ is smooth.
It meets any plane containing three points of $\Y$ just in these three points, for degree reasons.

\item When the camera center $a$ is coplanar with three of the points in $\X$, say $a\in\langle x_i, x_j, x_k\rangle$ then the twisted cubic containing $b$ 
degenerates into the union of a line and a conic. The conic is smooth unless $a\in\left(\langle x_i, x_j\rangle\cup\langle x_i, x_k\rangle
\cup\langle x_j, x_k\rangle\right)$. The singular point of this reducible curve is $\left(
\langle x_i, x_j, x_k\rangle\cap\langle x_p, x_q\rangle\right)$ with $\{i, j, k\}\cap\{ p, q\}=\emptyset$. Note that the $3$-fold
isomorphic to $\PP^2\times\PP^1$ given by $\langle x_i, x_j, x_k\rangle\times\langle y_p, y_q\rangle$ is contained in the centers-variety, it contains the three surfaces isomorphic to $\PP^1\times\PP^1$ given by $\langle \widehat{x_i}, x_j, x_k\rangle\times\langle y_p, y_q\rangle$,
$\langle x_i, \widehat{x_j}, x_k\rangle\times\langle y_p, y_q\rangle$,
$\langle x_i, x_j, \widehat{x_k}\rangle\times\langle y_p, y_q\rangle$, which consist of singular points of the centers-variety (see
Remark \ref{rem:sing5}).

\item When $a$ lies in the intersection of two planes spanned by three of the points, say $a\in\left(
\langle x_i, x_j, x_k\rangle\cap\langle x_i, x_p, x_q\rangle\right)$ 
with $\{i, j, k\}\cap\{ p, q\}=\emptyset$, then the twisted cubic containing $b$ 
degenerates into the union of three lines which are, respectively, 
$$\langle  y_j, y_k\rangle, \,\,\langle y_p, y_q\rangle,\,\, \textup{ and } 
\left(
\langle y_i, y_j, y_k\rangle\cap\langle y_i, y_p, y_q\rangle\right).$$
The singular points of the curve are two points as in the previous item.

\item When $a$ is collinear with two of the points in
$\X$, say $a\in\langle x_i, x_j\rangle$ then the twisted cubic containing $b$ degenerates into  the union of a line and a plane
which are respectively $\langle y_i, y_j\rangle$
and $\langle y_k, y_p, y_q\rangle$ with  $\{i, j\}\cap\{k, p, q\}=\emptyset$.

\item When $a$ is one of the points in
$\X$ then $b$ can be any point in $\PP^3$.
\end{itemize}
\end{remark}

\section{$n=6$}
\label{sec:n=6}
\begin{theorem}\label{thm:n=6locus}
Let $\X = \{x_1,\dots,x_6\}$ and $\Y = \{y_1, \dots, y_6\}$  be two generic sets of six ordered points in $\PP^3$. 
Their  
centers-variety is a surface $S \subset \PP^3 \times \PP^3$ that is isomorphic to a smooth quadric surface blown-up in six points. Equivalently, $S$ is isomorphic to the Del Pezzo surface of degree $2$ given by the plane blown-up in seven points\footnote{It is well known that its anticanonical bundle is ample but not very ample (see Remark \ref{rem:delpezzo2}).}.

If there are two cameras $A$ and $B$ with centers $a$ and $b$ such that $(Ax_i)_{i=1}^6 \sim (By_i)_{i=1}^6$, then $a$ and $b$ are constrained as follows. 
\begin{itemize}
\item The center $a$ lies on a smooth quadric surface $S_\beta  \subset\PP^3$ obtained as a projection of the centers-variety $S$. The surface $S_\beta$ contains $\X$ and the projection is the blow-up through $\X$. 

\item The center $b$ lies on a smooth quadric surface $S_\alpha  \subset\PP^3$ obtained as a projection of the centers-variety $S$.  The  surface $S_\alpha$ contains $\Y$ and the projection is the blow-up through $\Y$.
\end{itemize}
\begin{center}
\begin{tikzcd}
& S \subset \PP^3 \times \PP^3 \arrow{dl}{}[swap]{\pi_a} \arrow{dr}{\pi_b}{} & \\
S_\beta \subset \PP^3  & & S_\alpha \subset \PP^3
\end{tikzcd}
\end{center}
Each $a \in S_\beta$ corresponds to a unique $b \in S_\alpha$ such that $(a,b) \in S$. 
\end{theorem}

Recall from \ref{subsec:M26} that the coordinate ring $R_2^6$ of the moduli space $\M_2^6$ is generated by the six invariants $t_0, \ldots, t_5$ in \eqref{eq:degree 1 gens R26} and \eqref{eq:degree 2 gen R26}, and that 
\begin{align} \label{eq:igusa quartic relation}
    t_5^2 = F(t_0,\ldots, t_4).  
\end{align}
The polynomial $F$ 
has the formula
\begin{align}
    \textup{F} := (-t_2t_3 + t_1t_4 + t_0t_1 + t_0t_4 - t_0t_2 - t_0t_3 - t_0^2)^2 - 4t_0t_1t_4(-t_0+t_1-t_2-t_3+t_4), 
\end{align}
and its zero set is the quartic hypersurface in $\PP^4$ called 
the {\em Igusa quartic} or the {\em Segre quartic primal},  
By \eqref{eq:igusa quartic relation}, for each $(t_0,\ldots,t_4)$, there are two points $(t_0,\ldots,t_4, \pm t_5) \in \M_2^6$. This means that $\M_2^6$ is a double cover of $\PP^4$ branched over the Igusa quartic. 
By \eqref{eq:igusa quartic relation} and \eqref{eq:degree 2 gen R26}, 
$F=0$ is the same as
\begin{align}\label{eq:6onconic}
   t_5 =  [123][145][246][356]-[124][135][236][456]=0.
\end{align}
This is precisely the condition that six points in $\PP^2$ lie on a conic. Thus, there is a $2:1$ map 
\begin{align}\M_2^6 \,\,\stackrel{2:1}{\rightarrow} \,\,\PP^4.
\end{align}
A general fiber of this map consists of two points in $\M_2^6$; they represent associated pairs of $6$-tuples in $\PP^2$. The fiber of a point on the Igusa quartic corresponds to a $6$-tuple in $\PP^2$ that lies on a conic. Such $6$-tuples are self-associated. Indeed, note that for every bracket in the first term of  
\eqref{eq:6onconic} the complementary bracket appears in the second term. See \cite[Example 11.7]{DolgachevGIT} for these facts.

\subsection{A strategy to find the centers-variety} \label{subsec:strategy from flatland}
Using the generators of $R_2^6$, define the vector $t:=(t_0, \ldots, t_5)$. A generating set for $R_2^6$ is able to {\em separate} $\SL(3)$-orbits of six ordered points in $\PP^2$ in the following sense. If $\Pcal$ and $\Qcal$ are two sets of six ordered generic points in $\PP^2$ with representatives $P$ and $Q$, then $\Pcal \sim \Qcal$ if and only if 
\begin{align}
t(P) := (t_0(P), \ldots, t_5(P))= t(Q) := (t_0(Q), \ldots, t_5(Q)).
\end{align} 
 
We now describe a strategy used in \cite{flatlandpaper} to find the centers-variety. Suppose $\X \subset \PP^3$ is projected to $\Pcal \subset \PP^2$ through $a \in \PP^3\setminus\X$, and similarly, $\Y$ is projected to $\Qcal \subset \PP^2$ through $b\in \PP^3\setminus\Y$. For a bracket 
$[ijk]_p$ evaluated at $p_i,p_j,p_k \in \Pcal$, define $[ijka]_x$ to be the bracket evaluated at $x_i, x_j, x_k, a$, and set
\begin{align}
    t'(\X,a) := \left( \underbrace{[345a]_x[145a]_x[125a]_x[123a]_x[234a]_x}_{t_0'(\X,a)}, \,\,\,\cdots \,\,\,, \underbrace{[235a]_x[123a]_x[234a]_x[145a]_x^2}_{t_5'(\X,a)} \right).
\end{align}
We refer to $t_i'(\X,a)$ as a `lifting' of $t_i$; note that $t_i'(\X,a)$ is evaluated on $\X \cup \{a\}$ while $t_i$ is evaluated on 
$\Pcal$ where $\Pcal$ is the projection of $\X$ through $a$. 
Similarly, for $[ijk]_q$ evaluated at $q_i,q_j,q_k \in \Qcal$, define $[ijkb]_y$ to be the bracket evaluated at $y_i, y_j, y_k, b$, and set 
\begin{align}
    t'(\Y,b) = (t_0'(\Y,b), \cdots, t_5'(\Y,b)).
\end{align}

The following theorem  shows the usefulness of these `lifted' invariants. 

\begin{theorem} \cite[Theorem 16]{flatlandpaper} 
\label{thm:pull back tool}
For generic $\X$ and $\Y$, the projections of $\X$ through $a$ and $\Y$ through $b$ are $\SL(3)$-equivalent, i.e., correspond to the same point of $\M_2^6$, 
if and only if $t'(\X, a)= t'(\Y,b)$.
\end{theorem}

The equality $t'(\X, a)= t'(\Y,b)$ will give us the equations 
for the centers-variety. 

\subsection{A computational proof of Theorem~\ref{thm:n=6locus}}
We use the strategy in \S~\ref{subsec:strategy from flatland} to prove Theorem~\ref{thm:n=6locus}.
Fix a collection of six points $\X=\{x_1, 
\ldots, x_6\}$ in $\PP^3$ and set  
up the matrix 
\begin{align}
    M_\X(z) := \begin{bmatrix} x_1 & x_2 & x_3 & x_4 & x_5 & x_6 & z \end{bmatrix}
\end{align}
where $z = (z_0,z_1,z_2,z_3)^\top$ is a column of variables representing an unknown point in $\PP^3$. Lift each bracket polynomial $t_0, \ldots, t_5$ to 
$\PP^3$ by sending $[i\,j\,k] \mapsto [i\,j\,k\,7]$ (e.g. $t_0 = [123][456] \mapsto [1237][4567]$), and evaluate the lifted bracket polynomials on  the columns of 
$M_\X(z)$. Call these polynomials $q_0^\X(z), \ldots, q_5^\X(z)$. For $i=0, \ldots, 4$, 
$q_i^\X(z)$ is quadratic in $z_0,\ldots,z_3$ with coefficients depending on $\X$, while 
$q_5^\X(z)$ is a quartic in $z$ with coefficients determined by $\X$. 

Since the union of the indices in each of $t_0, \ldots, t_4$ is $\{1,\ldots,6\}$, 
$q_0^\X(z), \ldots, q_4^\X(z)$ will vanish if we set $z=x_i$ for any $i$. Therefore, $q_0^\X(z)=0, \ldots, q_4^\X(z)=0$ define five quadrics in $\PP^3$ containing $\X$. 
Since there can be at most four linearly independent 
quadrics through $\X$, there exists scalars 
$\alpha_0, \ldots, \alpha_4$ such that 
$\sum_{i=0}^4 \alpha_i q_i^\X(z) = 0$ for all $z \in \PP^3$. 
Repeating this process with   
$\Y = \{y_1, \ldots, y_6\} \subset \PP^3$ 
we obtain another five quadrics $q_0^\Y(z)=0, \ldots, q_4^\Y(z)=0$ containing $\Y$, and scalars $\beta_0, \ldots, 
\beta_4$ such that $\sum_{i=0}^4 \beta_i q_i^\Y(z) = 0$ for all $z \in \PP^3$. Let $S_\beta$ be the quadric surface in $\PP^3$ cut out by $\sum_{i=0}^4 \beta_i q_i^\X(z) = 0$, and $S_\alpha$ be the quadric surface in $\PP^3$ cut out by 
$\sum_{i=0}^4 \alpha_i q_i^\Y(z) = 0$.
Note the switchings!

\noindent{\em Proof of Theorem~\ref{thm:n=6locus}}:
Assume that $\X = \{x_1, \ldots, x_6\} \subset \PP^3$ and $\Y = \{y_1, \ldots, y_6\} \subset \PP^3$ are generic and $p_i = Ax_i$ and $q_i =By_i$ are the projections of $x_i$ and $y_i$ under two cameras $A,B$ with centers $a,b$ respectively. We can assume without loss of generality that the first five points in both $\X$ and $\Y$ are the standard points 
\eqref{eq:std vectors}.
Let $\Pcal = \{p_1, \ldots, p_6\} \subset \PP^2$ and 
$\Qcal = \{q_1, \ldots, q_6 \} \subset \PP^2$ and $t_i^\Pcal$ (respectively, $t_i^\Qcal$) denote the evaluation of $t_i$ at $\Pcal$ (respectively, $\Qcal$). If $\Pcal$ and $\Qcal$ correspond to the same point in $\M_2^6$ then  
$(t_0^\Pcal: \cdots : t_5^\Pcal) = (t_0^\Qcal: \cdots : t_5^\Qcal)$ 
which, by Theorem~\ref{thm:pull back tool}, is equivalent to 
\begin{align} \label{eq:equality of lifted polys}
(q_0^\X(a): \cdots : q_5^\X(a))= (q_0^\Y(b): \cdots : q_5^\Y(b)).
\end{align}     
In particular, $(q_0^\X(a): \cdots : q_4^\X(a)) = (q_0^\Y(b): \cdots : q_4^\Y(b))$ (dropping the last  coordinate on both sides in \eqref{eq:equality of lifted polys}), 
which implies that $\sum_{i=0}^4 \beta_i q_i^\X(a) = 0$ and $a$ lies on the quadric surface $S_\beta$ cut out by $\sum_{i=0}^4 \beta_i q_i^\X(z) = 0$ in $\PP^3$. Similarly,  
$\sum_{i=0}^4 \alpha_i q_i^\Y(b) = 0$ and $b$ lies on the quadric surface $S_\alpha$ cut out by 
$\sum_{i=0}^4 \alpha_i q_i^\Y(z) = 0$ in $\PP^3$.

Recall that each side of \eqref{eq:equality of lifted polys} lies in $\PP(1,1,1,1,1,2)$  and hence, 
\begin{align} 
\frac{q_i^\X(a)}{q_i^\Y(b)} = \frac{q_j^\X(a)}{q_j^\Y(b)}, \,\,\forall 0 \leq i < j \leq 4\\
\left(\frac{q_i^\X(a)}{q_i^\Y(b)} \right)^2 = \frac{q_5^\X(a)}{q_5^\Y(b)}, \forall i=0, \ldots, 4
\end{align} 
The $15$ equations above define an ideal $J$ in the variables $a$ and $b$. To remove the components at which $q_i^\X(a)$ and $q_i^\Y(b)$ vanish, we saturate $J$ with the product of the ideals 
\begin{align}
    \langle q_0^\X(a), \ldots, q_4^\X(a)\rangle \textup{ and } 
    \langle q_0^\Y(b), \ldots, q_4^\Y(b)) \rangle 
\end{align}
to obtain an ideal $I$ in $a,b$. The ideal $I$ has codimension $4$, degree $15$, and is generated by two quadrics, four cubics and $15$ quartics.  
The variety of $I$, a surface $S \subset \PP^3 \times \PP^3$, is the locus of $(a,b)$, i.e., the centers-variety. 
Eliminating $b$ from $I$, we obtain  $\langle \sum_{i=0}^4 \beta_i q_i^\X(a) \rangle$ and eliminating $a$ from $I$ we obtain $\langle \sum_{i=0}^4 \alpha_i q_i^\Y(b)\rangle$. 
Thus, the surface $S$ projects onto the surfaces $S_\beta$ and $S_\alpha$, sending 
$(a,b) \in S$ to $a \in S_\beta$ and $b \in S_\alpha$ via projections $\pi_a$ and $\pi_b$. 
\begin{center}
\begin{tikzcd}
& S  \subset \PP^3 \times \PP^3 \arrow{dl}{}[swap]{\pi_a} \arrow{dr}{}{\pi_b} & \\
S_\beta \subset \PP^3 && S_\alpha \subset \PP^3
\end{tikzcd}
\end{center}
All of the above claims can be checked in Macaulay2. \qed

\begin{remark}\label{rem:2112}
Embedding the surface $S$ in the Segre embedding of $\PP^3 \times \PP^3$ in 
$\PP^{15}$ yields a surface of degree $26$ generated by $75$ quadrics. The multidegree of $S$ in $\PP^3\times\PP^3$ is $(2, 11, 2)$. These assertions can also be verified in Macaulay2. 
\end{remark}

\begin{remark} {\bf A synthetic description.}
\label{rem:using the previous case}
There is always a way to get a sense of the answer to the Question for pairs of $n$ points by invoking the results for pairs of $(n-1)$ points as we now explain. 

If we pick 
$a \in \PP^3 \setminus \X$, and any five of the six points in $\Y$, then by Theorem~\ref{thm:n=5locus}, $b$ lies on a unique twisted cubic curve through these five $y$-points. Repeating this for all six subsets of five points in $\Y$, we obtain six twisted cubic curves in  $\PP^3$. When $n=6$,  the point $b$ must lie on all these twisted cubics, providing a necessary condition for the locus of $b$. Note that this logic does not tell us anything about the quadric surfaces $S_\beta$ and $S_\alpha$ that $a$ and $b$ lie on. 
However, given the answer for $n=6$, we can use the $n=5$ result to give a synthetic construction of the $b$-locus in the $n=6$ case. 
\begin{enumerate}
\item If $a\notin S_{\beta}$ then the six twisted cubics do not meet.

\item If $a\in S_{\beta}\setminus\X$ then the six twisted cubics meet in a point in $S_{\alpha}$ which is the $b$ such that $(a,b) \in S$.
In fact, if two of the six twisted cubics meet, then all six meet. Note that the individual twisted cubics are not contained in $S_\alpha$,
their intersection with $S_\alpha$ consists of six points, five of which are points in $\Y$. Therefore, the common intersection of the six twisted cubics can be checked by finding their sixth point of intersection with $S_\alpha$.

\item If $a$ is one of the six points in $\X$,  then there is only one twisted cubic, and it is contained in $S_\alpha$. One may wonder how we manage the points in $\X$, since the projection from $a\in\X$ is not defined. By Theorem \ref{thm:n=6locus} we get that $\X\subset S_\beta$, which  means that when
$a$ approaches $x_i\in\X$, from several directions along $S_\beta$, the corresponding points $b$ have a limit giving exactly the pairs $(x_i,\tilde b)$
with $\tilde b$  on a twisted cubic. In other words, 
 when $a$ approaches $x_i\in\X$, the six twisted cubics from 
 (2) have the same limit, which is again a twisted cubic!

\end{enumerate}
\end{remark} 

Having seen that the centers-variety $S$ is a quadric surface in $\PP^3 \times \PP^3$, we now give a more detailed description of it and show  precisely how it is related to the two smooth quadrics $S_\beta$ and $S_\alpha$. 

\begin{theorem}
\begin{enumerate}
\item The centers-variety $S$ is the image of a linear system of bidegree $(5,8)$ on a smooth quadric, passing with multiplicity three through six general points (denoted in standard notations as $5L_1+8L_2-3\sum_{i=1}^6E_i$). It has degree $26$ after composing with the Segre embedding $\PP^3\times\PP^3$ in $\PP^{15}$.
\item Moreover, $S$ is the closure of the graph of the birational map $S_\beta \to S_\alpha$ given by the linear system
$4L_1+7L_2-3\sum_{i=1}^6E_i$, where each $E_i$  (obtained by blowing-up $x_i$) is embedded as a twisted cubic and each of the six twisted cubics
$L_1+2L_2-\sum_{j\neq i} E_j$ is blown down to $y_i$.
\end{enumerate}
\end{theorem}

\begin{proof}
   The linear system in (1) is given by
    the line bundle on the smooth quadric surface $aL_1+bL_2-c\sum_{i=1}^6E_i$ with the nonnegative integers $a, b, c$ to be determined. The integer $c$ is the same for all six points by the symmetry of the construction. The fact that the exceptional divisors are twisted cubics implies that $c=3$. The degree $26$ seen in Remark \ref{rem:2112} implies that $2ab-6\cdot 9=26$ which implies that $ab=40$. By intersecting the system with the two lines on the quadric we get $a=5$, $b=8$, with the help of Macaulay2, see the ancillary file \verb|6pts-quadrics.m2| in the arXiv submission. 
    Item (2) is obtained from (1) by subtracting the hyperplane divisor $L_1+L_2$, this
    is typical of all graphs $(x, f(x))$ of a map $f$, since the identity component $x$ counts as a hyperplane divisor. The blow-up and blow-down of the six twisted cubics are guaranteed by the symmetry of the construction with respect to the two factors.
\end{proof}

 The surface $S$ is the blow-up of $S_\beta$ in the six points of $\X$ whose exceptional divisors are six twisted cubics in $S$, call the collection $\mathcal{T}_\X$. Similarly, $S$ is also the blow-up of $S_\alpha$ in $\Y$ whose exceptional divisors are another set of six twisted cubics, $\mathcal{T}_\Y$ on $S$. 
The birational map from $S_\beta$ to $S_\alpha$ first blows-up $S_\beta$ at $\X$ and then blows-down $\mathcal{T}_\Y$ to $S_\alpha$.
\begin{remark}
The six curves
$2L_1+L_2-\sum_{j\neq i} E_j$ on $S$ are blown down to six twisted cubics on $S_\alpha$. Each of them is a twisted cubic through five points among $\{y_1,\ldots, y_6\}$. Note that on a smooth quadric surface there are exactly two twisted cubics through five general points, of bidegree $(2,1)$ and $(1,2)$ respectively.
\end{remark}

\begin{remark}\label{rem:s'assoc}
  As a complement to Theorem \ref{thm:n=6locus}, the locus of $(a, b)$ for which there are two cameras $A,B$ with centers $a,b$ such that 
$(Ax_1,\ldots, Ax_6)$
is associated (in the sense of \S\ref{subsec:association}) to  $(By_1,\ldots, By_6)$ is again a surface $S'$ in $\PP^3 \times \PP^3$, isomorphic to a smooth quadric surface blown-up in six points. The surface $S'$ is the image of a linear system of bidegree $(2,5)$ on a smooth quadric, passing at six general points (denoted in standard notations as $2L_1+5L_2-\sum_{i=1}^6E_i$) and it has degree $14$ after composing with the Segre embedding $\PP^3\times\PP^3\to\PP^{15}$.
The surface $S'$ is the closure of the graph of the birational map $S_\beta\to S_\alpha$ given by the linear system
$L_1+4L_2-\sum_{i=1}^6E_i$, where each $E_i$ is embedded as a line and each of the six lines
$L_1-E_i$ are blown down to $y_i$.
\end{remark}

\begin{remark}\label{rem:delpezzo2} As an abstract surface, $S$ is isomorphic to $\PP^2$ blown up at seven points, hence it is a Del Pezzo surface.
If $H,F_1,\ldots, F_7$ is the standard basis of this blow-up,
with $H^2=1$ and $F_i^2=-1$,
the vocabulary is $L_1=H-F_1, L_2=H-F_2$, $E_1=H-F_1-F_2$, $E_i=F_{i+1}$
for $i=2,\ldots, 6$.
Hence $S$ is embedded in $\PP^3\times\PP^3$, composed with the Segre embedding, with the linear system 
$10H-2F_1-5F_2-3\sum_{i=3}^7F_i$.
The anticanonical bundle is $-K_S=3H-\sum_{i=1}^7F_i=2L_1+2L_2-\sum_{i=1}^6E_i$.
The Hilbert polynomial of $S$ with respect to the line bundle $\OO(1,1)$ is
$\chi(\OO_S(t))=13t^2+4t+1=26P_2(t)-35P_1(t)+10$.
The surface $S'$ of Remark \ref{rem:s'assoc} is isomorphic to $S$ 
and it is correspondingly embedded with the linear system $6H-F_1-4F_2-\sum_{i=3}^7F_i$.
\end{remark}

A further interesting study comes from self-associated $6$-tuples.
Any six points in $\PP^3$ determine a Weddle quartic surface, which is the locus of points that see the six points as lying on a conic. See \S~\ref{subsec:Weddle} for Weddle surfaces.
So we have two Weddle surface, one in each $\PP^3$ containing the centers $a$ and $b$. These Weddle surfaces are singular at the six points $\X$ (respectively, $\Y$), and contain the twisted cubic through the six points, which in general is not on the quadric $S_\beta$ ($S_\alpha$). Cutting with the quadrics, we get two singular curves of degree $8$ and bidegree $(4,4)$ on the Weddle surfaces.
Their genus is $4\cdot 4-4-4+1-6=3$, see \cite[exercise III 5.6 c]{Hart}. The two curves correspond to each other under the birational map from $S_\beta$ to $S_\alpha$.
This genus $3$ curve is desingularized in the surface $S\subset\PP^3\times\PP^3$, and it has class $(4L_1+4L_2-2\sum E_i)$. 
Its degree is indeed $(4L_1+7L_2-3\sum E_i)(4L_1+4L_2-2\sum E_i) = 16+28-36=8$ in $\PP^3$
and $(5L_1+8L_2-3\sum E_i)(4L_1+4L_2-2\sum E_i) = 20+32-36=16$ in $\PP^{15}$. It can be seen as the intersection
of the surface $S$ with the surface $S'$ of Remark \ref{rem:s'assoc}.

Let $H\subset\PP^3$ be a general hyperplane. We compute that
$\pi_a(\pi_b^{-1})(H\cap S_\alpha)$
is a curve of bidegree $(4,7)$ on $S_\beta$, which passes through each $x_i$ with multiplicity $3$. Moreover, if $p\in S_\alpha$ is a general point, we compute that $\pi_a(\pi_b^{-1})(p)$
is a unique point on $S_\beta$.
At the same time  $\pi_a(\pi_b^{-1})(y_i)$ is a twisted cubic for $i=1,\ldots, 6$.
It follows that $\pi_\alpha$ and $\pi_\beta$ are both blow-ups at finitely many points. 

A candidate for the linear system
embedding the blow-up of the quadric
$S_\beta$ is $L=5L_1+8L_2-3\sum_{i=1}^6E_i+R$ where $R=\sum n_jE_j$ is a residual supported at some points.
The equation $L^2=26$ implies
$40+40-54+R^2=26$ hence $R^2=0$, which implies $\sum n_j^2=0$ and eventually $R=0$. The surface $S$ can be identified with the image of $S_\beta$ under the linear system $5L_1+8L_2-3\sum_{i=1}^6E_i$. The other properties stated in Theorem \ref{thm:n=6locus} are now straightforward.

\section{$n=7$ and $n=8$}
\label{sec:n=7}

We now come to the last two cases we need to consider for the Question from the Introduction. 

\begin{theorem}
    \label{thm:n=7locus}
Let $\X = \{x_1,\dots,x_7\}$ and $\Y = \{y_1, \dots, y_7\}$  be two sets of seven generic ordered points in $\PP^3$. 
Their centers-variety consists of three distinct points.
\end{theorem}

As an immediate corollary we get that the centers-variety is generically empty when $n \geq 8$ which completes the answer to our Question. 

\begin{corollary}
    \label{cor:n=8locus}
Let $\X = \{x_1,\dots,x_8\}$ and $\Y = \{y_1, \dots, y_8\}$  be two generic sets of eight ordered points in $\PP^3$.
Then there are no two cameras $A,B$ such that 
$(Ax_i)_{i=1}^8 \sim (By_i)_{i=1}^8$. 
In other words, the centers-variety is empty. 
\end{corollary}

\begin{proof}
We have to prove that for general 
$\X = \{x_1,\ldots,x_8\}$ and  
$\Y = \{y_1,\ldots,y_8\}$ in $\PP^3$,
and for any $a\in\PP^3\setminus\X$, $b\in\PP^3\setminus\Y$, the projection of $\X$ through $a$ and the projection of $\Y$ through $b$ can never correspond to the same point on $\M_2^8$.

By Theorem~\ref{thm:n=7locus}, there are only three distinct $(a,b)$ that can project the subsets $\{x_1,\ldots,x_7\}$ and  
$\{y_1,\ldots,y_7\}$ to the same point in $\M_2^7$. Similarly, there are only three distinct $(a,b)$ that can project the subsets $\{x_2,\ldots,x_8\}$ and  
$\{y_2,\ldots,y_8\}$
to the same point in $\M_2^7$. Since the points $(a,b)$ depend on the data, generically they will not share a common $(a,b)$ and the corollary is proved.
 \end{proof}

To set a sense of what happens when $n=7$ we invoke the strategy explained in Remark~\ref{rem:using the previous case} that will provide a necessary condition for the structure of the centers-variety.
Suppose $a,b$ are the centers of cameras $A,B$ such that 
$(Ax_i)_{i=1}^7 \sim (By_i)_{i=1}^7$.
 By 
Theorem~\ref{thm:n=6locus}, each subset of six points in $\X$ constrains $a$ to lie on a quadric surface $S_\beta \subset \PP^3$. Thus when $n=7$, $a$ must lie at the intersection of seven quadric surfaces of the form $S_\beta$. Computing this intersection in Macaulay2 for random $\X$ and $\Y$ yields three points (which are the candidate $a$ points). These three points appear as the 
intersection of a twisted cubic with a hyperplane in $\PP^3$. 
By a similar calculation using $\Y$, we get three $b$ points. This shows that the centers-variety is zero-dimensional but it is not clear 
how to pick the pairs $(a,b)$ that form the centers-variety. In the rest of this section, we explain how to complete the analysis.

\subsection{The Goepel variety} \label{sec:goepel}
Previously, we answered the Question from the Introduction by working directly with the moduli space $\M_2^n$ and a generating set of invariants for its coordinate ring $R_2^n$. While several facts are known about $\M_2^7$ there seems to be no list of generators of $R_2^7$ in the literature, and finding a generating set would be an interesting problem.
However, there is a closely related variety to $\M_2^7$, called the {\em Goepel variety}, which will allow us to continue using invariant theory in the $n=7$ case. 
 We will also see that a subset of invariants that generate $R_2^7$ suffices to separate $\SL(3)$-orbits in $(\PP^2)^7$.

By \cite[Chapter II, Theorem 2]{dolgachev-ortland}, $\M_2^7$ is a smooth rational variety of dimension $6$. 
 In \cite[Chapter 7, \S7]{dolgachev-ortland}, Dolgachev and Ortland construct a birational isomorphism between $\M_2^7$ and the Goepel variety in $\PP^{14}$ which is parametrized by {\em Goepel functions}. This variety is also studied in 
\cite{ren-sam-schrader-sturmfels, Freitag_Salvati, BBFM} and we will use some of their computations and results.
In the notation of \S \ref{sec:Mdn}, 
$$R_2^7=\oplus_m (\underbrace{
\sym^{3m}\CC^{3}\otimes\ldots\otimes\sym^{3m}\CC^{3}}_{7\textrm{\ times}})^{\SL(3)},$$

while the coordinate ring of the Goepel variety is generated by the subring of $R_2^7$ given by the $m=1$ summand $\left(R_2^7\right)_1$. This is a proper subring of $R_2^7$  
spanned by the invariants of multidegree
$(3,\ldots, 3)$.  Indeed, as we will see in Lemma \ref{lem:Fano}, every invariant in $\left(R_2^7\right)_1$ vanishes on a configuration of seven points in which two coincide. However, such a configuration is stable by  
Corollary \ref{coro:stable7} and hence there exists an invariant
which is nonzero when evaluated on it and this invariant
cannot belong to the subring generated by $\left(R_2^7\right)_1$. 
Another difference between $R_2^7$ and $(R_2^7)_1$ is explored in Remark \ref{rem:invrings}. 
Nevertheless, the Goepel variety $\mathcal{G}$ is birational to $\M_2^7$.
As is common language in applied mathematics, the Goepel variety $\mathcal{G}$ can be seen as a ``relaxation'' of $\M_2^7$. This section is dedicated to describing the Goepel variety.

Consider a set of seven generic points in $\PP^2$ given by the columns of the following matrix:
\begin{align} \label{eq:7 generic pts}
M(x,y,z) := \begin{bmatrix}
      1&0&0&1&x_{0}&y_{0}&z_{0}\\
      0&1&0&1&x_{1}&y_{1}&z_{1}\\
      0&0&1&1&x_{2}&y_{2}&z_{2}
      \end{bmatrix}.
\end{align}
An invariant of these seven points in $\PP^2$ is the {\em Fano bracket polynomial} \cite[Eq. 6.6]{ren-sam-schrader-sturmfels}
\begin{align}\label{eq:Fano013}
f:=[124][235][346][457][156][267][137] \in (R_2^7)_1.
\end{align}

It is multihomogeneous in $(x, y, z)$ of multidegree $(3,3,3)$ (recall the first four points have been fixed). Permuting the seven indices we can generate $7!$ such Fano polynomials which reduce to only $30$ distinct polynomials modulo the $168$ element automorphism group of the Fano plane. Of the $30$, only $15$ are linearly independent.

Another invariant from $(R_2^7)_1$ is the {\em Pascal bracket polynomial} \cite[Eq. 6.8]{ren-sam-schrader-sturmfels}
\begin{align}
 g :=  [127][347][567] \left( [134][156][235][246]-[135][146][234][256] \right).
\end{align}
There are $105$ Pascal polynomials and only $14$ of them are linearly independent. 
The union of Fano and Pascal polynomials  contains only $15$ linearly independent elements. 

The $30+105=135$ Fano and Pascal polynomials are called Goepel functions.
The Goepel variety $\mathcal{G}$ is the image of the map 
\begin{align} \label{eq:original goepel}
    \varphi \,:\, (\PP^2)^7 \dashrightarrow \PP^{134}
\end{align}
parametrized by the $135$ Goepel functions. From the above observations, $\mathcal{G}$ lies in a $\PP^{14} \subset \PP^{134}$. 
By \cite[Chapter IX, \S7, Theorem 5]{dolgachev-ortland}, $\M_2^7$ is birationally equivalent to the Goepel variety $\mathcal{G}$, and hence, $\mathcal{G}$ is a $6$-fold in $\PP^{14}$. However, $\M_2^7$ is not isomorphic to  $\mathcal{G}$ since, as we will see, $\mathcal{G}$ is not smooth. 
Regardless, we can use the Goepel variety for our purposes. Indeed, we will work with a simpler parametrization than \eqref{eq:original goepel} to prove Theorem~\ref{thm:n=7locus}. 

For the permutation $\pi \in \mathfrak{S}_7$ and the Fano polynomial $f$ from \eqref{eq:Fano013}, let 
\begin{align}
\begin{split}
    f_\pi := & [\pi(1)\,\pi(2)\,\pi(4)]
[\pi(2)\,\pi(3)\,\pi(5)][\pi(3)\,\pi(4)\,\pi(6)][\pi(4)\,\pi(5)\,\pi(7)][\pi(1)\,\pi(5)\,\pi(6)]\\&[\pi(2)\,\pi(6)\,\pi(7)][\pi(1)\,\pi(3)\,\pi(7)]
\end{split}    
\end{align} 
denote the result of permuting the indices in the brackets of $f$ by $\pi$.
Among the $30$ distinct Fano polynomials, $15$ come from even permutations of $f$ and the remaining $15$ from odd permutations. We will use the $15$  Fano polynomials arising from even permutations. They are denoted with a 
$+$ superscript and listed below. This choice is ``tensor-friendly''
and simplifies computations.
\\
\begin{equation}\label{eq:15even}\begin{array}{ccccc}
f_0^+:= f_{1234567} & f_1^+ := f_{1234675} & f_2^+ := f_{1234756} & 
f_3^+ := f_{1235476} & f_4^+ := f_{1235647} \\
f_5^+ := f_{1235764} & f_6^+ := f_{1236457} & f_7^+ := f_{1236574} & 
f^+_8 := f_{1236745} & f_9^+ := f_{1237465}\\
f_{10}^+ := f_{1237546} & f_{11}^+ := f_{1237654} &
f_{12}^+ := f_{1243576} & f_{13}^+ := f_{1243657} & f_{14}^+ := f_{1243765}.
\end{array}
\end{equation}

Projecting $\mathcal{G}$ onto these $15$ coordinates, we obtain an isomorphic image of $\mathcal{G}$ that we will denote as $\mathcal{G}'$. In other words, $\mathcal{G}'$ is the image of the {\em Fano map}
\begin{align}\label{eq:phi'}
    \varphi' \,:\, (\PP^2)^7 \dashrightarrow 
    \PP^{14}
\end{align}
parametrized by $(f_0^+, \ldots, f_{14}^+)$. 
The indeterminacy locus of $\varphi'$ consists of $21$ components of codimension $2$, each corresponding to a coincident pair of the seven points, another $21$ components of codimension $3$  corresponding to five points being collinear,
and $35$ components of codimension $3$ corresponding to the condition that four  points are collinear and the remaining three are also collinear. Fixing four general points, we get only 
$12+3=15$ components of codimension $2$ ($12$ of degree $1$ and $3$ of degree $3$), corresponding to a coincident pair of the seven points, six components of codimension $3$ of the first kind, and $18$ components of codimension $3$ of the second kind. This can be checked computationally using Macaulay2, see the ancillary file \verb|fanos_even.m2|.

The following lemma is elementary.

\begin{lemma}\label{lem:Fano}
   Fano polynomials vanish on configurations of seven points in which two  coincide. 
\end{lemma}
\begin{proof}
    Consider $f=[124][235]346]457][156][267][137]$.
    Each pair $(i,j)$ with $1\le i<j\le 7$ belongs to one of the above seven  brackets. This property is retained under a permutation in $\mathfrak{S}_7$, and hence the result.
    A more geometric proof comes by inspection of the Fano matroid.
\end{proof}

Analogous to $f_0^+,\ldots f_{14}^+$, we have $f_0^-,\ldots f_{14}^-$,
corresponding to odd permutations.
It is well known that every cubic symmetric invariant of seven points vanishes, see \cite[end of \S 4]{Ottaviani_Sernesi}. Therefore, 
\begin{align} \sum_{i=0}^{14}f_i^++\sum_{i=0}^{14}f_i^-=0.
\end{align}
At the same time, there is a unique cubic skew-symmetric invariant of seven points, called the {\em Morley invariant}, which is (by the uniqueness)
\begin{align} 
M(p_1,\ldots, p_7)=\sum_{i=0}^{14}f_i^+-\sum_{i=0}^{14}f_i^-.
\end{align}
The two formulas above imply that
\begin{align} 
M(p_1,\ldots, p_7)=2\sum_{i=0}^{14}f_i^+. 
\end{align}
The vanishing of the Morley invariant corresponds to cutting the Goepel variety with the hyperplane in which the coordinates sum to zero. 
This vanishing has a 
geometric meaning, described in \cite[\S 4, \S 8]{Ottaviani_Sernesi}. It follows that $\sum_{i=0}^{14}f_i^+=0$ if and only if the corresponding seven points define a $2:1$ covering of $\PP^2$ which ramifies over a {\em Lueroth quartic}. Equivalently, the seven points are the eigenvectors of a ternary cubic for some quadratic form on $\CC^3$, see \cite[Coroll. 9.3]{Ottaviani_Sernesi} or \cite[\S 5]{AboSeigalSturmfels}.
Actually in \cite{AboSeigalSturmfels} it is shown that the variety of $7$-tuples of points that are eigenvectors of a smooth cubic has dimension $9$, but they considered a fixed quadratic form on $\CC^3$.

 By \cite[Theorem 5.1]{ren-sam-schrader-sturmfels}, $\mathcal{G}$ is a 
$6$-dimensional variety in $\PP^{14}$ of degree $175$. Its defining prime ideal is generated by $35$ cubic polynomials and $35$ quartic polynomials. Hence, the same is true for the vanishing ideal of $\mathcal{G}'$.  The parametrization used to define $\mathcal{G}$ in \cite{ren-sam-schrader-sturmfels} is using a different set of $15$ coordinates denoted $(r,s_\bullet, t_\bullet)$ which are related to $f_0, \ldots, f_{14}$ by a linear transformation. The cubics, quartics and the linear change of coordinates can all be found in the ancillary files at \url{https://arxiv.org/abs/1208.1229}.

All entries in the Jacobian of the cubics and quartics vanish identically at $36$ points on $\mathcal{G'}$ which means that the Jacobian has rank $0$ at these points. One such {\em highly singular} point is 
\begin{align}
    \omega = (1,1,1,1,1,1,1,1,1,1,1,1,1,1,1)
\end{align}
which is the image under $\varphi'$ of seven points in $\PP^2$ lying on a conic, see Proposition \ref{prop:fano image of projn from weddle}.
The remaining $35$ highly singular points come from the ${7 \choose 4} = 35$ ways of choosing four of the seven points to lie on a line. All of them have $3$ ones and $12$ zeros; an example being 
\begin{align} \label{eq:omega}
    (1,1,1,0,0,0,0,0,0,0,0,0,0,0,0 ).
\end{align}
(In contrast, the $7!$ permutations of seven generic points in $\PP^2$ yield 
$7!$ distinct images under the Fano map $\varphi'$ and no point among these is 
singular.)
Points on the line joining any pair of the $36$ highly singular points are also singular, but at these points, the Jacobian has rank $4$. Interestingly, these points do not have preimages in $(\PP^2)^7$. This implies that the image of $\varphi'$ is not closed.
In conclusion, we have found ${36\choose 2}$ lines consisting of singular points on the Goepel variety $\mathcal{G}'$.

\subsection{The computation of the centers-variety using $\mathcal{G}'$} \label{sec:(a,b) locus 7 pts}
Take two sets of seven points, $\X$ and $\Y$ in $\PP^3$, which we can assume without loss of generality to be the columns of the matrices
\begin{align} \label{eq:7ptsP3}
\begin{pmatrix}
    1 & 0 & 0 & 0 & 1 & r_0 & s_0\\
    0 & 1 & 0 & 0 & 1 & r_1 & s_1 \\
    0 & 0 & 1 & 0 & 1 & r_2 & s_2 \\
    0 & 0 & 0 & 1 & 1 & r_3 & s_3
\end{pmatrix} \,\,\,\,\textup{ and } \,\,\,\,
\begin{pmatrix}
    1 & 0 & 0 & 0 & 1 & r_0' & s_0'\\
    0 & 1 & 0 & 0 & 1 & r_1' & s_1' \\
    0 & 0 & 1 & 0 & 1 & r_2' & s_2' \\
    0 & 0 & 0 & 1 & 1 & r_3' & s_3'
\end{pmatrix}.
\end{align}
By Corollary~\ref{coro:indeterm}, we can project $\X$ through any $a \not \in \X$ to work with $\M_2^7$. However, since we are now working with the Goepel variety and not $\M_2^7$, we need to ensure that $a \not \in 2\X$, similar to the $n=5$ case. Indeed, if $a$ lies on the line joining two points in $\X$, then the projections of the two points coincide in $\PP^2$, and by Lemma~\ref{lem:Fano}, the map $\varphi'$ \eqref{eq:phi'} applied to image of $\X$ through $a$ is not defined. The description of the indeterminacy locus of the ideal of even Fano polynomials (after \eqref{eq:phi'}) and 
Corollary~\ref{coro:stable7} together ensure that there is no further condition to impose on $a$. 

Projecting $\X$ through $a \in \PP^3\setminus 2\X$ we obtain a configuration $\Pcal(a) = 
( p_1(a), \ldots, p_7(a))\in \PP^2$ where $p_i(a) = Ax_i$. They are the columns of the matrix
\begin{align}\label{eq:projected7pointsa}
        \begin{pmatrix}
      1&0&0&a_{0}&a_{0}-a_{3}&r_{3}a_{0}-r_{0}a_{3}&s_{3}a_{0}-s_{0}a_{3}\\
      0&1&0&a_{1}&a_{1}-a_{3}&r_{3}a_{1}-r_{1}a_{3}&s_{3}a_{1}-s_{1}a_{3}\\
      0&0&1&a_{2}&a_{2}-a_{3}&r_{3}a_{2}-r_{2}a_{3}&s_{3}a_{2}-s_{2}a_{3}
      \end{pmatrix}.
\end{align}
Each $p_i(a)$ is a function of $a$ with coefficients determined by $\X$. 
Similarly, projecting $\Y$ through $b \in \PP^3\setminus 2\Y$ we obtain a configuration $\Qcal(b)= ( q_1(b), \ldots, q_7(b) ) \in \PP^2$ where $q_i(b) = By_i$. They are the columns of the following matrix, 
each $q_i(b)$ a function of $b$ with coefficients determined by $\Y$.
\begin{align}
        \begin{pmatrix}
      1&0&0&b_{0}&b_{0}-b_{3}&r'_{3}b_{0}-r'_{0}b_{3}&s'_{3}b_{0}-s'_{0}b_{3}\\
      0&1&0&b_{1}&b_{1}-b_{3}&r'_{3}b_{1}-r'_{1}b_{3}&s'_{3}b_{1}-s'_{1}b_{3}\\
      0&0&1&b_{2}&b_{2}-b_{3}&r'_{3}b_{2}-r'_{2}b_{3}&s'_{3}b_{2}-s'_{2}b_{3}
      \end{pmatrix}.
\end{align}

Let $T_\X$ denote the subvariety of $\mathcal{G}'$ parametrized by $\varphi'(\Pcal(a))$ as $a$ varies in $\PP^3$. 
The ideal of $T_\X$ is obtained by  eliminating $a_0,a_1,a_2,a_3$ from 
\begin{align}
    \langle  t_i - f_i^+(\Pcal(a)) \rangle \subset \CC[a_0,\ldots,a_3,t_0,\ldots, t_{14}]
\end{align}
and we see that $T_\X$ is a $3$-fold in 
$\PP^{14}$ of degree $49$ cut out by $4$ linear polynomials, $6$ quadrics and $20$ cubics; 
in particular, $T_\X$ is contained in a $\PP^{10}$. 

Similarly, let  $T_\Y$ be the subvariety of $\mathcal{G}'$ parametrized by $\varphi'(\Qcal(b))$ as $b$ varies in $\PP^3$. It is also a $3$-fold lying in  a $\PP^{10}$ cut out by $4$ linear polynomials, $6$ quadrics and $20$ cubics. 

We are interested in pairs $(a,b) \in \left(\PP^3\setminus 2\X\right) \times
\left(\PP^3\setminus 2\Y\right)$ such that  $\varphi'(\Pcal(a)) = \varphi'(\Qcal(b))$.
These are precisely the points in $T_\X \cap T_\Y$. Computing this  intersection in Macaulay2, we obtain four distinct points. Three of them depend on the data ($\X, \Y$) and have multiplicity $1$ each, while the last one is the special point $\omega$ in \eqref{eq:omega} which has 
multiplicity $8$.
If $\varphi'(\Pcal(a)) = \varphi'(\Qcal(b)) = \omega$, then $\Pcal(a)$ and $\Qcal(b)$ 
lie respectively on a conic (isomorphic to $\PP^1$), and are in general not equivalent, since the $\SL(2)$-classes of $7$ points on $\PP^1$ have $4$ moduli. So $\omega$ has to be disregarded and only three pairs $(a,b)$ remain in the centers-variety.
We elaborate on these comments about $\omega$ below.

\subsection{The Weddle surface and curve in $\PP^3$.} \label{subsec:Weddle} 
The point $\omega$ can be explained by Weddle surfaces (seen already in \S~\ref{sec:n=6}) and curve, which we learned  from \cite[\S 2]{chiantini_weddlel}. A {\em Weddle surface} in $\PP^3$ is defined from six fixed generic points $z_1, \ldots, z_6$ 
in $\PP^3$ as follows. The space of all quadrics in $\PP^3$ that contain $z_1, \ldots, z_6$ is isomorphic to a $\PP^3$. Fix four independent quadrics $\mathcal{F}_0, \mathcal{F}_1, \mathcal{F}_2, \mathcal{F}_3$ in this space so that any  quadric through $z_1, \ldots, z_6$ can be written uniquely as 
$F^\lambda := \lambda_0 \mathcal{F}_0 + \lambda_1 \mathcal{F}_1 + \lambda_2 \mathcal{F}_2 + \lambda_3 \mathcal{F}_3$. Let 
$Q_i(z) := z^\top Q^i z$ be the quadratic form defining $\mathcal{F}_i$ with 
symmetric matrix $Q^i$. Then the symmetric matrix of the  quadratic form associated to $\mathcal{F}^\lambda$ is 
$Q^\lambda := \sum_{i=0}^3 \lambda_i Q^i$. 

Given any $z \neq z_i$ in $\PP^3$ there is a hyperplane of quadrics $\mathcal{F}^\lambda$ through $z$. 
Since the three quadrics in a basis of this hyperplane will intersect in eight points which include $z$ and  $z_1, \ldots, z_6$, there is an eighth point $z'$ that also lies on the three basis quadrics. 
Therefore,  we have the $2:1$ morphism
\begin{align}
    \eta \,:\, \PP^3 \mapsto \PP^3, \,\, z \mapsto \left(Q_0(z),\ldots, Q_3(z)\right).
\end{align}
The ramification locus of $\eta$ is the set of all $z$ such that $z=z'$ and is cut out by $\det J_\eta(z) = 0$ where $J_\eta$ is the Jacobian of $\eta$.
Since $J_\eta(z)$ is filled with linear forms in $z$, 
$\det(J_\eta(z))$ is a homogeneous quartic polynomial in $z_0,\ldots, z_3$ and the (hyper)surface it defines in $\PP^3$ is the  Weddle surface associated to $z_1, \ldots, z_6$. The branch locus of $\eta$ is the {\em Kummer surface} associated to $z_1, \ldots, z_6$. These are birationally equivalent quartic surfaces in $\PP^3$.

The Weddle surface associated to $z_1, \ldots, z_6$ has some well known properties which can be found in  
in \cite[Chapter 15, pp 170]{hudson-kummer-book} for instance. The $15$ lines joining any pair of points $z_i,z_j$ lie on the surface, the points $z_1, \ldots, z_6$ are the nodes of the Weddle surface, and the unique twisted cubic $C$ passing through $z_1, \ldots, z_6$ also lies on the Weddle surface and is contracted by the map $\eta$. The reason is that when $z\in C$ the three quadrics meet on $C$
and $z'$ is no longer defined. One can write a formula for the quartic defining the Weddle surface in terms of  $z_1, \ldots, z_6$. Assume without loss of generality that the six points are the columns of 
\begin{align}
\begin{pmatrix} 
1 & 0 & 0 & 0 & r_0 & s_0\\
0 & 1 & 0 & 0 & r_1 & s_1 \\ 
0 & 0 & 1 & 0 & r_2 & s_2 \\
0 & 0 & 0 & 1 & r_3 & s_3
\end{pmatrix}.
\end{align}
Then, by \cite[Chapter 15, pp 170]{hudson-kummer-book}, the quartic defining the Weddle surface associated to these points is cut out by the equation
\begin{align}
    \det \begin{pmatrix} 
    r_0s_0x_0^{-1} & x_0 & r_0 & s_0 \\
    r_1s_1x_1^{-1} & x_1 & r_1 & s_1 \\
    r_2s_2x_2^{-1} & x_2 & r_2 & s_2 \\
    r_3s_3x_3^{-1} & x_3 & r_3 & s_3 \\
    \end{pmatrix} = 0.
\end{align}

For us, the most important property is that the Weddle surface consists of the vertices of (singular) quadric cones through $z_1, \ldots, z_6$. Indeed, $z$ is on the Weddle surface if and only if $\rank (J_\eta(z)) < 4$ if and only if there exists $\lambda_0, \ldots, \lambda_3$ such that 
    $\sum \lambda_i z^\top Q^i = 0$ which is equivalent to 
    $0 = \sum \lambda_i Q^i z = Q^\lambda z$, 
    i.e., $z$ being the vertex of the singular quadric in the family with matrices $Q^\lambda$.
    
Now suppose we fix a collection $\X$ of seven generic points in $\PP^3$ as in  \eqref{eq:7ptsP3}. Then we obtain seven Weddle surfaces associated to each subset of six points in $\X$. Let $\mathcal{W}$ be the intersection of these seven Weddle surfaces. If $\mathcal{F}$ is a singular quadric that passes through all seven points in $\X$, then its vertex lies on every one of the seven Weddle surfaces and hence, on $\mathcal{W}$. Since the space of quadrics through seven points is isomorphic to a $\PP^2$ and the singularity condition imposes an additional constraint, we get that $\mathcal{W}$ is at least a curve. This must mean that it is indeed a curve since otherwise all seven Weddle surfaces must coincide.  Call $\mathcal{W}$ the {\em Weddle curve} associated to the $7$-tuple $\X$. In particular, every point on $\mathcal{W}$ is the vertex of a quadric that passes through $\X$. It is known that the Weddle curve has
genus $3$ and degree $6$, see for example \cite[\S 2.4]{chiantini_weddlel},\cite{chiantinietal}. 
Under the association $\M_3^7 = \M_2^7$,
the Weddle curve corresponds to the plane quartic curve which is the ramification locus of the $2:1$ map defined by the seven points, see \cite[\S 4.2]{Giorgio_quartics}. 

Suppose we now project the configuration $\X$ in \eqref{eq:7ptsP3} from a point $a$ on the Weddle curve associated to $\X$. Then $a$ is the vertex of a singular quadric that passes through all the points in $\X$ and hence the seven points in $\Pcal(a)$ shown in  \eqref{eq:projected7pointsa} will lie on a conic in $\PP^2$. We now compute the image of $\Pcal(a)$ under the Fano map $\varphi'$ in \eqref{eq:phi'}.

\begin{proposition} \label{prop:fano image of projn from weddle}
    If a configuration of $\mathcal{Q}$ of seven points  in $\PP^2$ lie on a conic, then $\varphi'(\mathcal{Q}) = \omega \in 
    \PP^{14}$. More generally, for any Fano polynomial $f_\pi$, 
    $f_\pi(\mathcal{Q}) = \textup{sign}(\pi) = \pm 1$. 
\end{proposition}
\begin{proof}
    Consider the Fano polynomial $f = [124][235][346][457][156][267][137]$ 
    and the adjacent transposition $(12) \in \mathfrak{S}_7$. We will first argue that 
$f(\mathcal{Q})/f_{(12)}(\mathcal{Q}) = \pm 1.$ This will imply that $f_\pi(\mathcal{Q})/ f(\mathcal{Q}) = \pm 1$ since 
$\pi$ is a product of adjacent transpositions. 
    First we compute that
    $f_{(12)} = -[124][135][346][457][256][167][237]$
    and hence 
    \begin{align}\label{eq:fanoratio}
        \frac{f}{f_{(12)}} = -\frac{[235][156][267][137]}{[135][256][167][237]}
    \end{align}
Note that the index $4$ is missing from \eqref{eq:fanoratio}. Since the points in $\mathcal{Q} \setminus \{q_4\}$ lie on a conic, they satisfy the bracket equation 
\begin{align}\label{eq:conic bracket reln}
  \frac{[123][156][257][367]}{[125][136][237][567]}=1
  \end{align}
This is the same as the identity in \eqref{eq:6onconic} where we have shifted the indices $4 \mapsto 5, 5 \mapsto 6, 6 \mapsto 7$ to account for the missing $4$. See for example, \cite[Proposition 23 (ii)]{OttavianiLectures} for this classical identity in invariant theory.
Check that the permutation $\tau = (237)$
will permute the left hand side of \eqref{eq:conic bracket reln} to $-\frac{f}{f_{(12)}}$ and hence $\frac{f}{f_{(12)}}=-1$. Therefore, 
if $\pi$ is an even permutation, 
then $\frac{f_\pi}{f} = 1$ which gives the result since the Fano map $\varphi'$ is made up of $15$ Fano polynomials corresponding to even permutations.
\end{proof}

\begin{remark}\label{rem:invrings} The fact that the Weddle curve is contracted to a (singular) point in $\mathcal{G}'$ by the map $\varphi'$
can be globalized by the following construction.
The subset of $\M_2^7$ consisting of configurations of points lying on a conic is easily seen to be isomorphic to $\M_1^7$.
We get an embedding $\M_1^7\subseteq \M_2^7$.
The birational  map $\M_2^7\to\mathcal{G}'$ contracts the subset $\M_1^7$ to the point $\omega\in\mathcal{G}'$ and shows another difference between $\M_2^7$ and $\mathcal{G}'$ and their invariant rings.
\end{remark}

\bibliographystyle{plain}
\bibliography{references}
\end{document}